\documentclass[11pt]{article}    
\usepackage[margin=1in]{geometry}
\usepackage{setspace}
\usepackage[utf8]{inputenc}
\usepackage[english]{babel}
\usepackage{verbatim}
\usepackage{amssymb}
\usepackage{mathtools}
\usepackage[title]{appendix}
\usepackage{float}
\usepackage{color,soul}
\usepackage{subcaption}
\usepackage{algorithm, algorithmic}
\usepackage{url}
\usepackage{booktabs}
\usepackage{tabularx}
\usepackage{amsthm} 
\usepackage[shortlabels]{enumitem}

\usepackage{float}  
\usepackage{subcaption}
\usepackage{thmtools}
\usepackage{thm-restate}
\newtheorem{claim}{Claim}
\newtheorem{insight}{Insight}
\newtheorem{corollary}{Corollary}[claim]
\newtheorem{lemma}{Lemma}

\usepackage{lscape}

\usepackage{natbib}
 \bibpunct[, ]{(}{)}{,}{a}{}{,}%

\usepackage{authblk}

\begin{document}

\title{Does Parking Matter? The Impact of Search Time for Parking on Last-Mile Delivery Optimization}

\author{\textsc{A Preprint}}




\maketitle

\begin{tabular}{ccc}
Sara Reed$^1$ & Ann Melissa Campbell$^2$ & Barrett W. Thomas$^2$ \\
\texttt{sara.reed@ku.edu} &  \texttt{ann-campbell@uiowa.edu} &  \texttt{barrett-thomas@uiowa.edu}
\end{tabular}

\begin{center}
\footnotesize
  $^1$ School of Business, University of Kansas, Lawrence, Kansas \par  
  \addvspace{\medskipamount}
  $^2$ Department of Business Analytics, University of Iowa, Iowa City, Iowa
\end{center}

\begin{abstract}
Parking is a necessary component of traditional last-mile delivery practices, but finding parking can be difficult. Yet, the routing literature largely does not account for the need to find parking. In this paper, we address this challenge of finding parking through the Capacitated Delivery Problem with Parking (CDPP). Unlike other models in the literature, the CDPP accounts for the search time for parking in the objective and minimizes the completion time of the delivery tour. When we restrict the customer geography to a complete grid, we identify conditions for when a Traveling Salesman Problem (TSP) solution that parks at each customer is an optimal solution to the CDPP. We then determine when the search time for parking is large enough for the CDPP optimal solution to differ from this TSP solution. We also identify model improvements that allow reasonably-sized instances of the CDPP to be solved exactly. We introduce a heuristic for the CDPP that quickly finds high quality solutions to large instances. Computational experiments show that parking matters in last-mile delivery optimization. The CDPP outperforms industry practice and models in the literature showing the greatest advantage when the search time for parking is high. This analysis provides immediate ways to improve routing in last-mile delivery.
\end{abstract}

\section{Introduction}

Recent studies find that delivery drivers spend on average 5.8 minutes and 24 minutes searching for each parking spot in Seattle and New York City, respectively \citep{Cruise_for_parking, Survey_Drivers}. In urban environments across the United States, a driver spends on average nine minutes searching for on-street parking \citep{INRIXParking}. Increasing levels of e-commerce leading to greater demands on delivery drivers further exacerbates the problem of parking. In New York City, more than 1.5 million packages need to be delivered every day \citep{NYPackages}. These volumes are expected to increase 68\% by 2045 \citep{IncreaseNY}. Attempts to improve productivity in last-mile delivery often explore the use of new technology, such as drones or autonomous vehicles \citep{DroneRouting,AutonomousGrid}. However, it may be many years before the deployment of such technologies, and we need to look for immediate ways to improve last-mile delivery to handle the exploding demand for such deliveries.

Despite the significant amount of time currently spent looking for parking and the growing delivery volumes that will increase search times further, the last-mile routing literature largely ignores the search time for parking in route planning. However, choosing when and where to park the vehicle is a choice made by the driver, often without the aid of decision support \citep{LastMile_Survey}. Where decision support is available, it is often in the form of a solution to the Traveling Salesman Problem (TSP) which does not consider parking. For example, a routing algorithm used by UPS includes a TSP algorithm, but gives drivers autonomy in making final routing decisions about where to drive, where to park, and where to walk \citep{UPS_Orion}. Only recently have modeling choices in the literature started to consider the need to park the vehicle and the trade-offs between driving and walking in the last-mile, but these studies are restricted to a set of tight assumptions \citep{Partitions, OptimizeServiceSets}.

In this paper, we examine how including the search time for parking impacts optimal routing decisions. We model last-mile delivery practices using the Capacitated Delivery Problem with Parking (CDPP). The CDPP is the problem of serving a set of customers with a vehicle and delivery person. The delivery person must park the vehicle to service customers on foot. The carrying capacity of the delivery person may be restricted based on the number, weight, or volume of packages. The delivery person must return to the parked vehicle after making deliveries on foot. Once returned to the vehicle, the delivery person can load more packages or drive to the next parking spot. Unlike other vehicle routing problems in the literature, the CDPP accounts for the search time for parking in the objective when minimizing the completion time of the delivery tour. In addition, the delivery person in the CDPP can serve multiple customer sets from the same parking spot. 

These generalizations relative to the current literature in last-mile delivery make the mixed integer programming (MIP) formulation for the CDPP difficult to solve. To analyze the impact of the search time for parking, we develop new technology to solve the CDPP to optimality. We exploit structure in optimal solutions to identify valid inequalities that raise the lower bound of the MIP and variable reduction techniques that reduce the large number of variables present in the model. These techniques allow us to solve larger instances than those present in the literature. For instances that face computational limitations, we provide a heuristic for the CDPP that finds high quality solutions quickly.  

To understand the impact of considering the search time for parking in routing decisions, we benchmark the CDPP with multiple other vehicle routing problems. The simplest comparison is to use the solution to the TSP with respect to driving times and assume the delivery person parks at every customer. When we restrict the customer geography to a complete grid, we identify conditions for when this TSP solution is an optimal solution to the CDPP. Then, we determine when the search time for parking is large enough for the CDPP optimal solution to differ from this TSP solution. For experimental comparisons, we benchmark the CDPP with models that represent current industry practice as well as recent models in the literature. 

The CDPP was introduced as a benchmark in \cite{AutonomousGrid}. This paper is the first to fully analyze and explore the problem. The contributions of this paper can be summarized as follows:
\begin{itemize}
\item We analyze the first model to include the search time for parking in the objective function to evaluate the impact of the search time for parking on last-mile delivery optimization.
\item By restricting the customer geography to a complete grid, we identify conditions under which following a TSP tour of the customers and parking at each one is an optimal solution to the CDPP as well as the value of the search time for parking that defines change in the structure of an optimal solution to the CDPP.
\item We contribute to solving the CDPP exactly through valid inequalities and variable reduction. 
\item For instances that face computational limitations, we provide a heuristic solution to the CDPP that finds high quality solutions quickly. 
\item We provide valuable insights from computational experiments showing when the consideration of parking in the model makes the greatest impact. 
\end{itemize}

Section \ref{LitReview} reviews the literature specifically addressing the limited work in vehicle routing that considers the  need to park the vehicle. Section \ref{Model} discusses the service times, assumptions, and the MIP formulation for the CDPP. In Section \ref{TSP Analysis}, we directly compare the structure of the CDPP solution with a TSP solution that parks at every customer on a complete grid of customers, providing conditions for when the solutions are equivalent and the value of the search time for parking that defines change in the structure of an optimal solution to the CDPP. Section \ref{ModelImprovements} discusses necessary improvements to the model to be able to solve reasonably-sized instances. Section \ref{Experimental Design} provides the experimental design. In Section \ref{Computational Performance}, we show the effect of the model improvements on computational performance. Section \ref{Experimental Results} presents the experimental results and discusses the impact of considering the search time for parking on the structure of the solution and the completion time of the delivery tour. Section \ref{Heuristic} provides a heuristic  for the CDPP that finds high quality solutions quickly. Conclusions and future work are discussed in Section \ref{Conclusions_FutureWork}.

\section{Literature Review} \label{LitReview}

In this section, we focus on the vehicle routing literature that consider the need to park the vehicle. We begin by summarizing work that introduced the CDPP to the literature. The remaining studies in routing of last-mile delivery do not explicitly consider the search time for parking in the objective function. Instead, they focus on the trade-offs of walking and driving for the delivery person and ignore trade-offs between the difficulty to find parking and other routing decisions.

\cite{AutonomousGrid} introduce the CDPP as a benchmark to the Capacitated Autonomous Vehicle Assisted Delivery Problem (CAVADP). To evaluate the impact of autonomous-assisted delivery in urban environments, \cite{AutonomousGrid} solve the CDPP on a complete grid. However, some instances could not be solved to optimality so the potential sets of customers were restricted to reduce the number of variables. An optimal solution to the CDPP on a complete grid is shown to follow a linear trend in the number of customers. \cite{AutonomousGeneral} use the CDPP as a benchmark when exploring the CAVADP on a general graph to represent urban to rural settings. These papers use the CDPP as a benchmark and do not explore the CDPP beyond stating and modeling the problem. In this paper, we present model improvements to the CDPP and focus on the impact of including the search time for parking in the objective function. The conclusions of this paper support the use of the CDPP as a benchmark problem to represent delivery practices where the delivery person must park the vehicle prior to servicing customers on foot. Further, while the CAVADP shows great promise for autonomous-assisted delivery, it may be many years before the deployment of such technologies. The analysis of this paper provides immediate ways to improve routing in last-mile delivery.

The closest problem to the CDPP is the two-echelon last-mile delivery system introduced by \cite{OptimizeServiceSets}. In this system, the decisions are the locations where the vehicle will park, the locations visited by the delivery person on foot, and the order of delivery locations in both the driving and walking routes. The carrying capacity of the delivery person restricts the volume and weight of packages in a customer set. \cite{OptimizeServiceSets} also require that one node within each customer set be designated as the parking location. A generalization for the CDPP relative to \cite{OptimizeServiceSets} is the  ability to serve multiple customer service sets from the same parking spot. In addition, this paper provides a model formulation for the CDPP that allows the parking locations to differ from customer locations. \cite{OptimizeServiceSets} include the clustering of customers as a decision in the optimization problem, but do not include the search time for parking in the objective function. Instead, the objective is a weighted sum of the driving time and walking time for the delivery person. This paper uses the objective function of \cite{OptimizeServiceSets} to benchmark the objective function in the CDPP. Comparing the solutions with the objective value in \cite{OptimizeServiceSets} to the solutions using an objective value that includes the search time for parking highlights the impact of the search time for parking on optimal routing decisions. \cite{OptimizeServiceSets} use a branch-and-cut algorithm to solve instances up to 30 customers. This paper introduces model improvements for the CDPP to solve larger instances than the instances in \cite{OptimizeServiceSets}.

\cite{Partitions} define a two-level clustered routing problem to distinguish between the driving route and walking routes of the delivery person. Service time to customers is restricted by time windows. The grouping of customers is an input to the model. Two different partitions of customers are considered: one based on observations of drivers from a case study in London and one based on geographical proximity. Each cluster is required to have a parking location within the cluster. The CDPP generalizes this approach by making the partition of customers and the parking locations optimization decisions. In addition, the CDPP allows multiple service sets from the same parking spot. The optimization decisions in \cite{Partitions} are to select the parking location in each cluster, route the vehicle between the parking locations and depot, and route the walking of the delivery person in each cluster. These decisions capture the delivery person's walk back to the parking location but fail to capture the advantages of serving more than one customer set per parking location if the search time for parking is high.  \cite{Partitions} solve the mixed integer programming formulation for the case study in London and observe that optimizing with respect to time windows can reduce total operation time. However, the search time for parking is not considered in the operation time. We leave the consideration of time windows in the CDPP for future work and focus this paper on the impact that the search time for parking has on the total time of the delivery tour and the structure of the solution.

More generally, the CDPP can be related to the two-echelon routing problem. In the application of last-mile delivery, one echelon refers to the driving route and the second echelon refers to the walking routes of the delivery person. \cite{TwoEchelon} provide a survey on two-echelon routing problems. The single truck-and-trailer routing problem with satellite depots (STTRPSD) is similar to the CDPP with the trailer representing the vehicle and the truck representing the delivery person. The truck and trailer are routed on a subset of the satellites (i.e. parking locations for the CDPP) and then the customers are visited from the truck (i.e. delivery person) with routes at each satellite. \cite{STTRPSD} present a branch-and-cut algorithm to solve the STTRPSD. \cite{STTRPSD} solve all instances up to 50 customers to optimality and solve 100 to 200 customers to an average optimality gap of 3.02\%. The test instances consider at most 10 satellites with 25 to 50 customers and at most 20 satellites with 100 to 200 customers.  In our case, we consider all customer locations to be available parking spots (i.e. satellite locations) significantly increasing the size of the model. Therefore, we introduce necessary model improvements to be able to solve the CDPP to optimality. These model improvements may have application in the truck-and-trailer routing problem.

Unlike the STTRPSD and other two-echelon models for last-mile delivery, the CDPP considers a search time in finding a parking location. This feature aligns the CDPP with the two-echelon capacitated location-routing problem (2E-CLRP) where an opening cost is associated with the satellites (i.e. parking locations). \cite{LRP-2E} present three mixed integer programming formulations for the 2E-CLRP. The 2E-CLRP considers two different fleets of vehicles for first-level and second-level trips, connected by the satellites for transshipment operations. Therefore, the 2E-CLRP can decompose into two different capacitated location routing problems \citep{2E-CLRP_decomposition}. However, in the CDPP, there exists a dependence between the two levels. The vehicle must remain at the parking spot on the first-level route while the delivery person serves potentially multiple second-level routes on foot. In addition, the travel times in the two levels differ and we capture these differences by using real-world data for the driving times and walking times between customers. For the application of last-mile delivery, the CDPP balances the trade-offs of the delivery person walking and driving to find a new parking location, so differing travel times influence the solution structure. \cite{LRP-2E_GRASP} and \cite{LRP-2E_metaheuristic} propose metaheuristic approaches to solve the 2E-CLRP. The heuristic solution proposed in this paper may be applicable to other problems that align with the 2E-CLRP where there is a dependence between the two echelons.

Another line of research uses simulation to model parking availability and the impacts on commercial vehicle parking behavior \citep{Microscopic_Simulation, Traffic_Microsimulation}. \cite{Figliozzi_Tipagornwong} combine queuing and logistical models to model parking availability but use continuous approximation models to estimate routing constraints. These papers show that parking has an impact on operations, but still do not consider parking in the routing optimization. In this paper, we argue the search time for parking makes a significant impact on the routing optimization. 

\section{CDPP} \label{Model}

In this section, we detail the service times, assumptions, and MIP for the CDPP. Section \ref{ServiceTimes} defines the problem and provides notation for the parameters and service times. Then, Section \ref{MIP} provides the MIP formulation.

\subsection{Problem Description and Notation} \label{ServiceTimes}

The CDPP serves a set of $n$ customers in a set $C$ by a delivery person with a vehicle. The delivery person and vehicle start and end the tour at the depot, denoted as $0$. The delivery person must park the vehicle to service customers on foot. Let $\Pi$ be the set of parking locations and depot. The search time for parking at parking location $i \in \Pi\setminus \{0\}$ is $p_i$ minutes. Once parked, the delivery person services a set of customers on foot.  After servicing a customer set on foot, the delivery person returns to the parked vehicle. The delivery person can serve another set of customers from this parking spot or move to a different parking spot (e.g. $k \in \Pi$) incurring a search time for parking of $p_{k}$ minutes.

Each customer $i$ requires a single package delivery with weight $\psi_i$ and volume $v_i$. Multiple customer nodes at the same customer location may represent a single customer ordering multiple packages. Let $S$ be the set of potential customer service sets. The capacity of the delivery person may restrict the number of packages, weight, or volume of the service sets in $S$. For example, if the delivery person can service at most $q$ packages at a time, then $|\sigma_j|\leq q$ for all sets $\sigma_j \in S$. Similarly, if the delivery can service at most $\Psi$ in weight or $V$ in volume, then $\sum_{i \in \sigma_j} \psi_i \leq \Psi$ or $\sum_{i \in \sigma_j} v_i \leq V$, respectively, for all sets $\sigma_j \in S$. For each $i\in C$, let $J_i = \{\sigma_j \in S |i \in \sigma_j\}$ be the set of customer service sets that includes customer $i$. Define $I_{ij} = 1$ if $\sigma_j \in J_i$ for all $i \in C$ and $\sigma_j \in S$, and 0 otherwise. Table \ref{ParametersDescription} summarizes the parameters for the CDPP. 

\begin{table}[h]
\renewcommand{\arraystretch}{0.75}
\begin{tabularx}{\linewidth}{ll} 
\toprule
 \textbf{Notation} & \textbf{Description} \\ 
  \midrule
$n$ & Number of customers \\
$C$ & Set of customer locations \\
$S$ & Set of customer service sets \\
$\Pi$ & Set of parking locations and depot \\
$p_i$ & Expected search time for parking at parking location $i \in \Pi$ (minutes)\\
$q$ & Capacity of delivery person (number of packages)\\
$\Psi$ & Weight capacity of delivery person \\
$\psi_i$ & Weight of package for customer $i \in C$ \\
$V$ & Volume capacity of delivery person\\
$v_i$ & Weight of package for customer $i \in C$ \\
$J_i$ & Set of customer sets that include customer $i$ for $i \in C$ \\
$I_{ij}$ & Indicator variable that customer $i$ is in service set $\sigma_j$ for $i \in C, \sigma_j \in S$ \\
  \bottomrule 
\end{tabularx}
  \caption{Set of parameters.} \label{ParametersDescription}
  \end{table}

The service times require more detail in their definition. Table \ref{CDPPCostsDescription} summarizes the service times for the CDPP. Let $D(i,k)$ be the time to drive between locations $i$ and $k$ for $i,k \in \Pi$. We assume driving times satisfy the triangle inequality. Let $d_{ik}$ be the time to drive from location $i$ to location $k$ and park at $k$. Then, $d_{ik} = D(i,k) + p_k$ for $i,k \in \Pi$ such that $i \neq k$. In the case where $k=0$ (i.e. the return to the depot), the vehicle does not need to search for parking and $d_{i0} = D(i,0)$ for all $i \in \Pi \setminus \{0\}$. 

Let $W(i,k)$ be the time to walk between locations $i$ and $k$ for $i,k \in C \cup \Pi \setminus \{0\}$. We assume walking times satisfy the triangle inequality. Let $w_{ij}$ be the shortest walking time to service set $\sigma_j$ when parked at parking location $i$ for $i \in \Pi \setminus \{0\}, \sigma_j \in S$. This walking time is the shortest walk from parking location $i$ to the first customer to be served in set $\sigma_j$, the walk between customers in $\sigma_j$, and the walk back to customer $i$ where the vehicle is parked. For the pair $(i,\sigma_j) \in \Pi \setminus \{0\} \times S$, let $(c_1,c_2,...,c_{|\sigma_j|})$ be an optimal order to serve $\sigma_j$ when parked at $i$. Then, $w_{ij} = W(i,c_1) + W(c_1,c_2) + \cdots + W(c_{|\sigma_j|-1}, c_{|\sigma_j|}) + W(c_{|\sigma_j|}, i)$. 

Let $f_j$ be the time to load package(s) to service set $\sigma_j$. We consider a loading time linearly dependent  on the number of packages in the service set, i.e. $f_j = f\cdot |\sigma_j|$ for some $f \geq 0$. Therefore, the total loading time in the delivery tour is a constant $nf$, and the solution to the CDPP is equivalent to the solution when $f_j \equiv 0$ or $f=0$. 

\begin{table} [h]
\renewcommand{\arraystretch}{0.75}
\begin{tabularx}{\linewidth}{ll} 
\toprule
  \textbf{Notation} & \textbf{Description} \\ 
  \midrule
$D(i,k)$ & Time to drive from $i$ to $k$ for $i,k \in \Pi$ (min) \\
$d_{ik}$ & Time to drive from $i$ to $k$ and park at $k$ for $i \in \Pi, k \in \Pi \setminus \{0\}$ such that $i \neq k$ (min) \\
$d_{i0}$ & Time to drive from parking location $i$ to depot for $i \in \Pi\setminus \{0\}$ (min)\\
$W(i,k)$ & Time to walk from $i$ to $k$ for $i,k \in C \cup \Pi \setminus \{0\}$ (min) \\ 
$w_{ij}$ & Time to walk and serve set $\sigma_j$ while parked at location $i$ for $i \in \Pi \setminus \{0\}$, $\sigma_j \in S$ (min) \\
$f$ & Time to load one package (min)\\
$f_j$ & Time to load packages for customer set $\sigma_j$ for $\sigma_j \in S$ (min)\\
  \bottomrule 
\end{tabularx}
  \caption{Definition of service times in CDPP.} \label{CDPPCostsDescription}
  \end{table}

\subsection{MIP Formulation} \label{MIP}

\cite{AutonomousGrid} introduce a variant of the CDPP presented here and provide a MIP formulation. \cite{AutonomousGrid} take $f_j \equiv g$ for some constant $g\geq0$. For the purposes of clarity and updated notation, we provide the general formulation here. We also adapt single commodity subtour elimination constraints, as opposed to the adapted MTZ subtour elimination constraints presented in \cite{AutonomousGrid}.  Table \ref{VariablesCDPP} summarizes the decision variables for the model. 

\begin{table} [h]
\renewcommand{\arraystretch}{0.75}
\begin{tabularx}{\linewidth}{ll} 
\toprule
  \textbf{Notation} & \textbf{Description} \\ 
  \midrule
$x_{ik}$ & $x_{ik}=1$ if the vehicle drives from $i$ to $k$ and parks at parking location $k$ (if $k \neq 0$) for  \\
& $i,k \in \Pi$ such that $i \neq k$\\
$y_{ij}$ & $y_{ij} =1$ if the delivery person is parked at parking location $i$ and serves set $\sigma_j$ \\
& for $i \in \Pi$ and $\sigma_j \in S$ \\
$v_{ik}$ & Flow of packages from location $i$ to location $k$ for $i \in \Pi$ and $k \in \Pi \setminus \{0\}$ such that $i \neq k$ \\
  \bottomrule 
\end{tabularx}
  \caption{Set of decision variables in CDPP.} \label{VariablesCDPP}
  \end{table} 

\begin{align}
\min & \sum_{i\in \Pi} \sum_{k\in \Pi \setminus \{i\}} x_{ik}d_{ik}+ \sum_{i \in \Pi \setminus \{0\}} \sum_{\sigma_j \in S} y_{ij}(w_{ij} + f_j) \label{CDPPobj}\\
\textrm{s.t.} & \sum_{i\in \Pi} x_{0i} = 1 \label{CDPPdepotFROM}\\
& \sum_{i\in \Pi} x_{i0} = 1 \label{CDPPdepotTO}\\
& \sum_{k\in \Pi \setminus \{0\}} \sum_{\sigma_j \in J_i} y_{kj} = 1  && \forall i \in C \label{CDPPMember}\\
& \sum_{k \in \Pi} \setminus \{i\} x_{ki} = \sum_{k \in \Pi\setminus \{i\}} x_{ik} && \forall i \in \Pi \setminus \{0\} \label{CDPPComeLeave} \\
& y_{ij} \leq \sum_{k \in \Pi \setminus \{i\}} x_{ki} && \forall i \in \Pi\setminus \{0\} , \sigma_j \in S \label{CDPPVisit}\\
\label{Alternative_Subtour1}
&\sum_{i \in \Pi \setminus \{0\}} v_{0i} = n && \\  \label{Alternative_Subtour2}
&v_{ik} \leq n\cdot x_{ik} && \forall i \in \Pi, k \in \Pi \setminus \{0\} \textrm{ s.t. } i \neq k \\ \label{Alternative_Subtour3}
 &\sum_{k \in \Pi \setminus \{i\}} v_{ki} - \sum_{k \in \Pi \setminus \{0,i\}} v_{ik} = \sum_{\sigma_j\in S}|\sigma_j| y_{ij} && \forall i \in \Pi \\ \label{Alternative_Subtour4}
&v_{ik} \in \mathbb{Z}_+ && \forall i \in \Pi, k \in \Pi \setminus \{0\} \textrm{ s.t. } i \neq k \\
& x_{ik} \in \{0,1\} && \forall i,k \in \Pi \textrm{ s.t. } i \neq k \label{CDPPx} \\
& y_{ij} \in \{0,1\} && \forall i \in \Pi \setminus \{0\}, \sigma_j \in S\label{CDPPy} 
\end{align}

\noindent The objective function in Equation \eqref{CDPPobj} minimizes the completion time of the delivery tour. The first term includes the driving time and search time for parking. The second term is the walking and service time for the delivery person. Constraints \eqref{CDPPdepotFROM} and \eqref{CDPPdepotTO} require the vehicle to leave from the depot and return to the depot, respectively. Constraints \eqref{CDPPMember} require that each customer is served in a service set. When the delivery person parks at a parking location, Constraints \eqref{CDPPComeLeave} ensure that the vehicle leaves that parking location. Given that the delivery person services set $\sigma_j$ when parked at parking location $i$ (i.e. $y_{ij}=1$), Constraints \eqref{CDPPVisit} require the vehicle to drive to and park at location $i$. Constraints \eqref{Alternative_Subtour1}-\eqref{Alternative_Subtour4} provide the adapted single commodity subtour elimination constraints. Constraints \eqref{Alternative_Subtour1} and \eqref{Alternative_Subtour2} ensure $n$ packages flow through the network.  While at parking location $i$, the flow should change by the number of customers in all  service sets served from parking spot $i$. Constraints \eqref{Alternative_Subtour3} capture the change in flow. The integer constraints on the $v_{ik}$ variables are given in Constraints \eqref{Alternative_Subtour4}. Finally, the binary constraints on variables $x_{ik}$ and $y_{ij}$ are given in Constraints \eqref{CDPPx} and \eqref{CDPPy}, respectively.

\section{When is the TSP optimal?} \label{TSP Analysis}

A TSP solution parking at every customer location is a feasible solution to the CDPP. Because driving is often faster than walking, this TSP solution may be optimal if the search time for parking is low, particularly when customers are further apart. However, when customers are close together, a small search time for parking may make it advantageous for the delivery person to consolidate packages into larger customer service sets to reduce the number of times the delivery person searches for parking. Therefore, we expect the density of customer locations in combination with the search time for parking to impact the structure of the CDPP solution.

In this section, we identify when a TSP solution parking at every customer location is optimal for the CDPP. We analyze the trade-offs between the density of customers and the search time for parking by considering a complete grid of customers and determining when the search time for parking is large enough that an optimal solution to the CDPP is no longer this TSP solution. For the purpose of this analysis,  we restrict the setting to a $\sqrt{n} \times \sqrt{n}$ complete grid of $n$ customers where $\sqrt{n}$ is even. We assume $\Pi = C \cup 0$ and take $p_k =p$ for all $k \in \Pi \setminus \{0\}$. Let $\hat{d}$ be the time to drive a unit and $\hat{w}$ be the time to walk a unit where $\hat{d} \leq \hat{w}$. The length of a block is $\hat{l}$ units. The capacity of the delivery person is based on the number of packages $q$ assuming that each customer has a single package. Let $(0,0)$ represent the location of the depot with the bottom left corner of the grid at $(1,1)$, bottom right corner at $(\sqrt{n},1)$, top left corner at $(1,\sqrt{n})$, and top right corner at $(\sqrt{n}, \sqrt{n})$.

We first contemplate the case of a TSP tour through the set of customer locations parking at every customer. Figure \ref{6_6_Grid} shows an example of this TSP solution on a $6\times 6$ grid of customers. The black square indicates the location of the depot and circles represent the customer locations. The solid blue lines represent the path of the vehicle. A red square designates the customer location as a parking spot. Lemma \ref{TSP_obj_value} characterizes the completion time of a TSP tour through the grid parking at every customer. The result follows from the analysis in \cite{AutonomousGrid}.

\begin{lemma} \label{TSP_obj_value}
Consider a TSP solution where the delivery person parks at every customer location. The objective value for this solution on a $\sqrt{n} \times \sqrt{n}$ complete grid of customers when $\sqrt{n}$ is even is 
\begin{align} \label{TSP_obj_value_eq}
(2 \cdot MinDistance + n)\hat{d}\hat{l} + nf + np
\end{align}
where $MinDistance = \min_{c\in C} D(0,c)$ is the minimum distance between the depot and grid.
\end{lemma}

\begin{figure}[h]
  \centering
  \includegraphics[scale = 0.3]{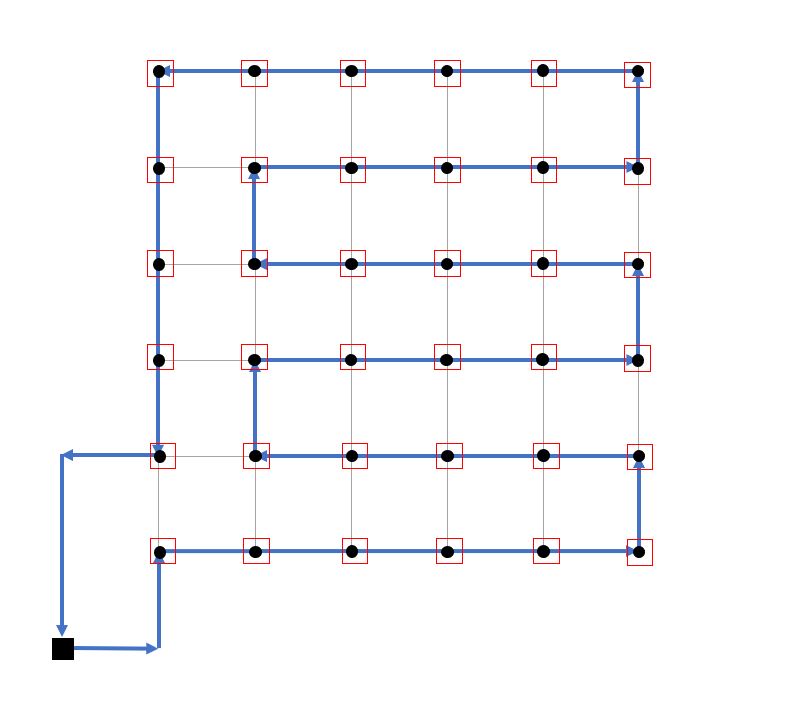}
  \caption{An example of the TSP solution parking at every customer on a $6 \times 6$ grid of customers. }
\label{6_6_Grid}
\end{figure}

Next, we identify nonzero search times for parking where an optimal solution to the CDPP is this TSP solution parking at every customer location. We also determine when $p$ becomes large enough that this TSP solution is not optimal. Claim \ref{Grid_q2} summarizes these results when $q\leq 2$. This and all other proofs can be found in Appendix \ref{Proofs}.

\begin{claim} \label{Grid_q2}
Assume $q \leq 2$. Then,
\begin{enumerate}[(a)]
\item if $p \leq \hat{l}(2\hat{w}-\hat{d})$, then an optimal solution to the CDPP is a TSP solution parking at every customer;
\item if $p > \hat{l}(2\hat{w}-\hat{d})$, then a TSP solution parking at every customer is not an optimal solution to the CDPP.
\end{enumerate}
\end{claim}

Increasing the capacity of the delivery person to $q=3$ packages allows for gains from consolidating customers into service sets at lower search times for parking than identified in Claim \ref{Grid_q2}.  Claim \ref{Grid_q2} finds a threshold of $p=\hat{l}(2 \hat{w}-\hat{d})$ for the solution structure of the CDPP to change. When $q=3$, Claim \ref{Grid_q3} reduces this threshold to $p=\hat{l}(\frac{4}{3}\hat{w}-\hat{d})$.

\begin{claim} \label{Grid_q3}
Assume $q=3$. Then,
\begin{enumerate}[(a)]
\item if $p \leq \hat{l}(\frac{4}{3}\hat{w}-\hat{d})$, then an optimal solution to the CDPP is a TSP solution parking at every customer;
\item if $p > \hat{l}(\frac{4}{3}\hat{w}-\hat{d})$, then a TSP solution parking at every customer is not an optimal solution to the CDPP.
\end{enumerate}
\end{claim}

Now, we discuss the implications of Claims \ref{Grid_q2} and \ref{Grid_q3} on optimal routing decision for last-mile delivery. In both claims, the value of $\hat{l}$, and thus the density of customers, plays a crucial role in determining the structure of the optimal solution to the CDPP. For this analysis, we consider grids representing urban-to-rural settings. We estimate the parameters from the instances presented in \cite{AutonomousGeneral}. We take $\hat{l}=0.07$ miles and $\hat{l} = 0.29$ miles for urban and rural environments, respectively. We estimate $\hat{d}=12.5$  min/mi and $\hat{w}=20$ min/mi based on the driving time and walking time, respectively, in \cite{AutonomousGrid}.

Figure \ref{Cutoff_p} shows the value of the search time for parking $p$ that defines change in the structure of an optimal solution to the CDPP when $q\leq2$ and $q=3$. In urban environments, when customers are closer together, small values for the search time for parking impact the structure of the CDPP solution. If $p>1$ minute, a TSP solution parking at every customer location is not an optimal solution to the CDPP. In an empirical study, \cite{Cruise_for_parking} find that delivery drivers spend on average 5.8 minutes searching for each parking spot in Seattle. This parking time suggests that productivity gains can be achieved in urban environments by considering the search time for parking in routing decisions. In lower density areas, like rural environments, Figure \ref{Cutoff_p} shows that the search time for parking must be higher to change the structure of the optimal solution from a TSP solution parking at every customer. We also expect the driving speed in rural areas to be greater than in urban areas. A higher driving speed produces higher values of $p$ in Claims \ref{Grid_q2} and \ref{Grid_q3} further supporting that a TSP solution parking at every customer remains optimal for higher values of $p$ in low density areas. In Section \ref{Experimental Results}, our experimental results indicate that a TSP solution parking at every customer best approximates optimal solutions to the CDPP in rural environments.

\begin{figure}[h]
  \centering
  \includegraphics[scale = 0.3]{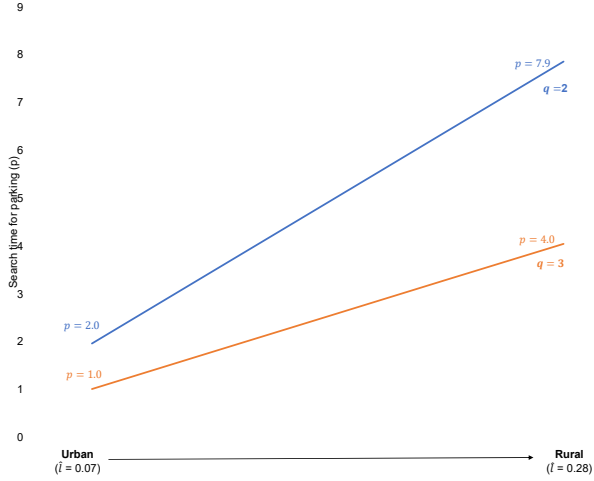}
  \caption{The value of the search time for parking $p$ that determines the optimality of the TSP solution parking at every customer location for the CDPP. }
\label{Cutoff_p}
\end{figure}

\section{Model Improvements} \label{ModelImprovements}

In this section, we develop a series of inequalities and variable reduction techniques that allow us to much more effectively solve the CDPP. Section \ref{Computational Performance} evaluates the impact of these model improvements on computational performance. Appendix \ref{AdditionalModelImprovements} provides model improvements that apply to alternative assumptions regarding the loading time function $f_j$.

The delivery person must park the vehicle in order to service customers on foot. The parking locations and what customer sets to serve from each parking spot are decisions in the optimization problem. If the delivery person parks at customer location $i \in \Pi \cap C$, Claim \ref{park_set} identifies that one of the customer sets will include customer $i$. Recall that the solution to the CDPP when $f_j = f \cdot |\sigma_j|$ for some $f\geq 0$ is equivalent to when $f = 0$. To simplify the proofs in this section, we show the results for $f=0$.

\begin{claim} \label{park_set}
If the delivery person parks at $i \in \Pi \cap C$ (i.e. $x_{ki} = 1$ for some $k \in \Pi \setminus \{i\}$), then there exists $\sigma_j \in J_i$ such that $y_{ij} = 1$, i.e. 
\begin{align}
\sum_{k \in \Pi\setminus \{i\}} x_{ki} = \sum_{j \in J_i} y_{ij}  && \forall i \in \Pi \cap C.
\end{align}
\end{claim}

Claim \ref{park_set} shows that when the vehicle parks at the location of customer $i \in \Pi \cap C$, one of the service sets at that parking location includes customer $i$. In particular, there exists an optimal solution where customer $i$ is served alone while parked at the location of customer $i$. Claim \ref{f0_Single} formalizes this observation.
 
\begin{claim} \label{f0_Single}
For $i \in \Pi \cap C$, let $\sigma_{j_i} = \{i\}$. There exists an optimal solution where $y_{ij_i} = 1$ for each parking location $i \in \Pi \cap C$ (i.e. $x_{ki} = 1$ for some $k \in \Pi \setminus \{i\}$), i.e. 
\begin{align} \label{f0_Single_Equation}
\sum_{k \in \Pi \setminus \{i\}} x_{ki} = y_{ij_i} && \forall i \in \Pi \cap C.
\end{align}
\end{claim}

To strengthen Claim \ref{f0_Single}, we sum over all customers in set $\Pi \cap C$ in Equation \eqref{f0_Single_Equation}. If $\Pi = C$, then we conclude the number of parking spots is equal to the number of singleton customers serviced while parked at its location. Corollary \ref{parkingspots_f0} provides the strengthened result.
\begin{corollary}
\label{parkingspots_f0} 
\begin{align} \label{parkingspots_f0_equation}
\sum_{i \in \Pi \cap C} \sum_{k \in \Pi \setminus \{i\}} x_{ki} = \sum_{i \in \Pi \cap C} y_{ij_i}.
\end{align}
\end{corollary}

We also use Claim \ref{f0_Single} to identify variables that will not be used in an optimal solution and may be removed from the MIP in Section \ref{MIP}. If the vehicle parks at customer $i$, Claim \ref{f0_Single} concludes $\sigma_{j_i} = \{i\}$ is serviced. Constraints \eqref{CDPPVisit} restrict each customer to be in exactly one service set. Therefore, any other set that includes $i$ will not be serviced while the vehicle is parked at customer $i$.  Corollary \ref{f0_y0} formalizes this result.

\begin{corollary} \label{f0_y0}
For all $i \in \Pi \cap C$, $y_{ij} = 0$ for all $\sigma_j \in \{\sigma_j \in J_i\vert \textrm{ } |\sigma_j| \geq 2\}$.
\end{corollary}

\section{Experimental Design} \label{Experimental Design}

In this section, we introduce the test instances and experimental design. The integer programming model for the CDPP is implemented in Python 3.7.0 using the Gurobi 9.0.0 solver with a 32 thread count on the University of Iowa's Argon high performance computing cluster \citep{Argon}.

To explore the impact of the search time for parking in all customer geographies, we use the test instances for the case study of Illinois that represent urban to rural settings as described in \cite{AutonomousGeneral}. In particular, we use Cook County, Adams County, and Cumberland County to represent urban, suburban, and rural environments, respectively, based on the classification of the United States Department of Agriculture \citep{UrbanRural}. We evaluate ten instances of $n=50$ customers for each county. The test instances include real-world data for the driving times and walking times between customers \citep{TestInstances}. To test the computational performance on larger instances, we generate five instances in each of Cook, Adams, and Cumberland counties for $n=100$ following the procedure of \cite{AutonomousGeneral}. All test instances are posted at: https://doi.org/10.25820/data.006124. 

We restrict the parking locations of the vehicle to all customer locations (i.e. $\Pi = C \cup \{0\}$.) We vary the search time for parking in each type of customer geography (i.e. urban, suburban, and rural) as parking poses a greater challenge in urban environments than rural environments. Within each geography, we consider a search time for parking $p$ independent of the parking location (i.e. for all $k \in \Pi \setminus \{0\}$, $p_k = p$ for some $p\geq0$.) Doing so allows us to isolate the impact of the search time for parking on optimal routing decisions in different geographies. Like \cite{AutonomousGeneral}, we use $p=9$ minutes in Cook County, $p=5$ minutes in Adams County, and $p=1$ minute in Cumberland County as our base case. These values reflect location-dependent parking times where a higher search time for parking is realized in urban environments compared to rural environments. For each county, we also experiment with smaller values of $p$ to understand how including the search time for parking changes the structure of the solution in the CDPP. For Cook County, we experiment with $p=0, 3$ and $6$ minutes. For Adams County, we include experiments with $p=0$ and $3$ minutes. For Cumberland County, we also consider $p=0$. 

We choose the parameter values for the capacity of the delivery person $q$ and the time to load a package $f$ based on observations in the literature. A study of last-mile delivery in London finds that a delivery person delivers 3 packages on average per stop \citep{Observations_London}. Therefore, we consider a carrying capacity based on the number of packages $q$ and use $q=3$ packages in our base case. We experiment with capacities of 1 to 4 packages. The average delivery time for a package is estimated to be between 2.5--4.1 minutes \citep{Observations_London, DeliveryTime_Rome, DeliveryTime_DHL}. This time often includes multiple delivery activities. \cite{Observations_London} observe 4.1 minutes of service time per customer where this estimate includes the time spent unloading the package, walking to the customer, and gaining proof-of-delivery. \cite{Partitions} estimate a one minute walking time allowance and one minute consignee service time. Combining the estimates of \cite{Observations_London} and \cite{Partitions}, we assume the time to unload a package is 2.1 minutes in our base case. 

\section{Computational Performance} \label{Computational Performance}

In this section, we show the impact of the model improvements in Section \ref{ModelImprovements} on computational performance. For this analysis, we use the first five instances of $n=50$ customers for each county.  We consider the base case ($q=3$ packages and location-dependent parking times). Since we assume a loading time linearly dependent on the number of packages, the total loading time is a constant and it is sufficient to solve the MIP with $f=0$. 

First, we focus on the valid inequalities introduced in Claim \ref{f0_Single} and Corollary \ref{parkingspots_f0}. Table \ref{ValidInequalities} provides the average linear program (LP) relaxation bound at the root node and the average runtime (in minutes). The second column of Table \ref{ValidInequalities}, labeled \textit{LP}, gives the average LP bound for the MIP in Section \ref{MIP}. The third column in Table \ref{ValidInequalities}, labeled \textit{Claim \ref{f0_Single}}, gives the average LP bound of the MIP including the valid inequalities in Claim \ref{f0_Single}. The results show that Claim \ref{f0_Single} raises the LP bound on average 5.3\%, 7.8\%, and 3.7\% for Cook, Adams, and Cumberland counties, respectively. Corollary \ref{parkingspots_f0} relies on the results of Claim \ref{f0_Single}. On average, the fourth and fifth columns of Table \ref{ValidInequalities}, labeled \textit{Corollary \ref{parkingspots_f0}} and \textit{Both} respectively, show no further increase in the LP bound by including the valid inequality presented in Corollary \ref{parkingspots_f0}.

\begin{table}[h]
\footnotesize
\renewcommand{\arraystretch}{0.75}
\centering
\begin{tabularx}{\linewidth}{p{25mm} rrrrrrrr}
\toprule
&  \multicolumn{4}{c}{\textbf{LP Bound}} & \multicolumn{4}{c}{\textbf{Runtime (min)}} \\
\cmidrule(lr){2-5}\cmidrule(lr){6-9} 
\textbf{County} & \textbf{LP} & \textbf{Claim \ref{f0_Single}} & \textbf{Corollary \ref{parkingspots_f0}} & \textbf{Both} &\textbf{MIP} & \textbf{Claim \ref{f0_Single}} & \textbf{Corollary \ref{parkingspots_f0}} & \textbf{Both} \\
\midrule
Cook & 128.2 & 135.0 & 135.0 & 135.0 & 88.2 & 20.8 & 35.2 & 19.8\\
Adams & 114.8 & 123.7 &123.7 & 123.7& 68.3 & 18.5 & 22.6 & 18.0 \\
Cumberland & 63.4 & 65.7 & 65.7 & 65.7& 39.1 & 11.2 & 16.1& 11.2\\ 
\bottomrule
\end{tabularx}
\caption{LP bound and runtime (minutes) using the valid inequalities in Section \ref{ModelImprovements} in the base case and $f=0$.} \label{ValidInequalities}
\end{table}

The valid inequalities significantly reduce the runtime.  The sixth column of Table \ref{ValidInequalities}, labeled \textit{MIP (min)}, gives the average runtime for the MIP in Section \ref{MIP}. The seventh and eighth columns in Table \ref{ValidInequalities} give the average runtime of the MIP including the valid inequalities in Claim \ref{f0_Single} and Corollary \ref{parkingspots_f0}, respectively. The results show that Claim \ref{f0_Single} improves the runtime on average 76.4\%, 72.9\%, and 71.3\% for Cook, Adams, and Cumberland counties, respectively. Corollary \ref{parkingspots_f0}  improves the runtime on average 60.1\%, 64.4\%, and 58.8\% for Cook, Adams, and Cumberland counties, respectively. The eighth column in Table \ref{ValidInequalities} shows the average runtime of the MIP using both valid inequalities. On average, used together, Claim \ref{f0_Single} and Corollary \ref{parkingspots_f0} improve the runtime on average 77.5\%, 73.6\%, and 71.4\% for Cook, Adams, and Cumberland counties, respectively, relative to the MIP presented in Section \ref{MIP}. In summary, these valid inequalities significantly improve the run time across all counties, and present a greater impact on urban instances that have a higher search time for parking.

Next, we discuss the variable reduction technique identified in Corollary \ref{f0_y0}. Table \ref{VariableReduction} provides details on the number of $y_{ij}$ variables present in the model with and without the variable reduction identified in Corollary \ref{f0_y0} for various capacities of the delivery person $q$. Let $Y = \{y_{ij}|i \in C, \sigma_j \in S\}$ be the set of $y_{ij}$ variables for a given instance of the model presented in Section \ref{MIP}. Recall that we consider all service sets of size at most $q$ packages. The second column of Table \ref{VariableReduction} gives $|Y| = n\sum_{i=1}^q \binom{n}{q}$. For each $i \in C$, define $\bar{J}_i =\{\sigma_j \in J_i\vert \textrm{ } |\sigma_j| \geq 2\}$ to be the service sets identified by Corollary \ref{f0_y0} that will not be serviced while parked at customer $i$. Then, let $\bar{Y}_i = \{y_{ij} | \sigma_j \in \bar{J}_i\}$ represent the variables identified by Corollary \ref{f0_y0} that can be removed from the model with respect to parking at customer $i$. In total, Corollary \ref{f0_y0} removes the variables $\bar{Y} = \cup_{i \in C} \bar{Y}_i$ from the model. Let $\hat{Y} = Y \setminus \bar{Y}$ be the remaining variables in the model. The third and fourth columns of Table \ref{VariableReduction} give the remaining number of variables $|\hat{Y}|$ and the percent reduction in $y_{ij}$ variables. In the base case of $q=3$ packages, Corollary \ref{f0_y0} reduces the number of $y_{ij}$ variables by 5.9\%. This reduction in variables also reduces the number of Constraints \eqref{CDPPVisit}. As capacity $q$ increases, the percent reduction in $y_{ij}$ variables increases. However, the number of $y_{ij}$ variables $|\hat{Y}|$ grows at a much greater rate making it more difficult to solve instances of larger capacity. We further discuss the computational limitations with respect to capacity in Section \ref{Heuristic}.

\begin{table}[h]
\renewcommand{\arraystretch}{0.75}
\centering
\begin{tabular}{lrrr}
\toprule
\textbf{\boldmath $q$} &\multicolumn{1}{c}{\boldmath \textbf{$|Y|$}} & \multicolumn{1}{c}{\boldmath \textbf{$|\hat{Y}|$}} & \textbf{\boldmath Percent Reduction in $y_{ij}$ Variables} \\
\midrule
 1 & 2,500 & 2,500 & 0.0\% \\
 2 & 63,750 & 61,300 & 3.8\% \\
 3 & 1,043,750& 982,500 & 5.9\%\\
 4 & 12,558,750 & 11,576,300 & 7.8\%\\
\bottomrule
\end{tabular}
\caption{Variable reduction results of Corollary \ref{f0_y0} for $n=50$ and various capacities $q$ (packages).} \label{VariableReduction}
\end{table}

Table \ref{VariableReduction_Runtime} provides the average runtime (in minutes) of the MIP with and without the variable reduction technique in Corollary \ref{f0_y0}. The second column of Table \ref{VariableReduction_Runtime} gives the average runtime for the MIP in Section \ref{MIP} using all variables in $Y$.  The third column of Table \ref{VariableReduction_Runtime} gives the average runtime under the reduced $y_{ij}$ variables in $\hat{Y}$. Reducing the number of $y_{ij}$ variables improves the runtime on average  47.8\%, 69.9\%, and 61.7\% for Cook, Adams, and Cumberland counties, respectively. The fourth column of Table \ref{VariableReduction_Runtime} gives the average runtime under the reduced $y_{ij}$ variables in $\hat{Y}$ and the valid inequalities (i.e. Claim \ref{f0_Single} and Corollary \ref{parkingspots_f0}). Including the variable reduction technique further improves the impact of the valid inequalities on run time.  Reducing the number of $y_{ij}$ variables further improves the runtime on average 25.4\%, 26.3\%, and 35.0\% for Cook, Adams, and Cumberland counties, respectively, relative to the MIP with valid inequalities (presented in the ninth column of Table \ref{ValidInequalities}). In summary, Corollary \ref{f0_y0} significantly improves the runtime across all customer geographies. 

\begin{table}[h]
\renewcommand{\arraystretch}{0.75}
\centering
\begin{tabular}{lrrr}
\toprule
&\multicolumn{3}{c}{\textbf{Runtime (min)}} \\
\cmidrule(lr){2-4}
\textbf{County} & \textbf{\boldmath $Y$\& no VIs} & \textbf{\boldmath $\hat{Y}$ \& no VIs} & \textbf{\boldmath $\hat{Y}$ \& VIs} \\
\midrule
Cook & 88.2 & 46.0 & 14.8\\
Adams & 68.3 & 20.5 & 13.3 \\
Cumberland & 39.1 & 15.0 & 7.3 \\ 
\bottomrule 
\end{tabular}
\caption{Computational results from the variable reduction of Corollary \ref{f0_y0} in the base case and $f=0$.} \label{VariableReduction_Runtime}
\end{table}

\section{Experimental Results} \label{Experimental Results}
In this section, we explore the impact of including the search time for parking in optimal routing decisions on the structure of the solution and the completion time of the delivery tour. Section \ref{Benchmarks} introduces benchmarks to the CDPP that reflect current industry practice as well as recent models in the literature. Section \ref{Experimental_Results_Summary} summarizes the differences between the CDPP and benchmark solutions in the base case (location-dependent parking times, $q=3$ packages, and $f=2.1$ minutes). Then, Sections \ref{ParkingTimeP} and \ref{CapacityQ} discuss the impact of the search time for parking and capacity of the delivery person, respectively, on the structure of the CDPP solution.

\subsection{Benchmarks} \label{Benchmarks}

We introduce three benchmarks to the CDPP and use these benchmarks to highlight why including the search time for parking matters in routing optimization for last-mile delivery. 

\subsubsection{No parking time}

For this benchmark, we solve the CDPP with $p=0$ to show how including the search time for parking changes the structure of the solution. When we restrict the customer geography to a complete grid of customers, Claims \ref{Grid_q2} and \ref{Grid_q3} show an optimal solution to the CDPP when $p=0$ is a TSP solution parking at every customer. On a general customer geography, \cite{Dissertation} shows that this conclusion holds if driving between customers is always faster than walking. Let $v$ be the optimal value of the CDPP when $p=0$ and $s$ be the number of times the vehicle parks in the respective optimal solution. Then, the completion time with a solution of this benchmark is $v + sp$ minutes. 

\subsubsection{Modified TSP}

The routing algorithm for UPS provides the driver a solution based on a TSP algorithm \citep{UPS_Orion}. The driver has autonomy to make the final routing decisions including where the delivery person will walk and where he/she will drive. To model this real-life practice, we use a TSP solution to fix the order of service. To reflect choices made by the driver and allow service to customers on foot, we transform this TSP solution to take into account the trade-offs between the search time for parking, driving time, and walking time.

This benchmark, hereafter called the Modified TSP, is a route-first cluster-second method to optimize the trade-offs of walking, driving, and searching for parking given a fixed customer order. A TSP solution with respect to driving times fixes the order of customer service. We generate potential service sets based on this order and restrict the size of the service set based on the capacity of the delivery person $q$. In addition, we allow the delivery person to serve multiple customer service sets from the same parking spot while maintaining the order of customer service. We implement this benchmark  by restricting the sets of driving variables $x_{ik}$ and service variables $y_{ij}$ in the CDPP. We update the walking service time $w_{ij}$ to maintain the order of service. We use the objective function in Equation \eqref{CDPPobj} to minimize the completion time of the delivery tour and capture the impact of the search time for parking on the driver's decision making.

\subsubsection{Relaxed M-S}

One of the limitations in the models presented by \cite{Partitions} and \cite{OptimizeServiceSets} is that both require a parking location in every customer set. Unlike \cite{Partitions}, \cite{OptimizeServiceSets} include the clustering of customers as a decision in the optimization problem. To benchmark the CDPP with the literature, we consider a relaxed version of the model by \cite{OptimizeServiceSets}, hereafter called Relaxed M-S, that allows the delivery person to serve multiple customer service sets from the same parking spot. For comparison purposes, we define the carrying capacity of the delivery person in the Relaxed M-S benchmark to be based on the number of packages $q$. 

The key difference between the Relaxed M-S and the CDPP is the objective function. The objective function for the CDPP, in Equation \eqref{CDPPobj}, includes the search time for parking when evaluating the completion time of the delivery tour. \cite{OptimizeServiceSets} consider a weighted sum of the driving and walking times for the delivery person. Using the notation of Section \ref{ServiceTimes}, Equation \eqref{MS_generalized_obj} presents the objective function in \cite{OptimizeServiceSets},
\begin{align} \label{MS_generalized_obj}
\alpha\sum_{i \in \Pi} \sum_{k \in \Pi \setminus \{i\}} x_{ik}D(i,k) + (1-\alpha)\sum_{i \in \Pi}\sum_{\sigma_j \in S} y_{ij}w_{ij}
\end{align}
where $\alpha \in [0,1]$. 

We implement the Relaxed M-S benchmark by using Equation \eqref{MS_generalized_obj} as the objective value in the CDPP. Since we assume a loading time linearly dependent on the number of packages, the total loading time is a constant and added to the objective value in Equation \eqref{MS_generalized_obj}. Similar to \cite{OptimizeServiceSets}, we consider $\alpha \in \{0.6, 0.8\}$ in the Relaxed M-S benchmark. Let $v$ be the optimal value of the Relaxed M-S benchmark (including loading time) and $s$ be the number of times the vehicle parks in the respective solution. Then, the completion time with a solution of this benchmark is $v + sp$ minutes. Comparing the completion times for the CDPP and the Relaxed M-S benchmark highlights the impact of including the search time for parking in the objective function.

\subsection{Comparison of CDPP to Benchmarks} \label{Experimental_Results_Summary}

In this section, we compare the completion time of the delivery tour in the CDPP with the values of the solutions for the three benchmarks in the base case (location-dependent parking times, $q=3$ packages, and $f=2.1$ minutes).

Figure \ref{BenchmarkComparison_BaseCase} shows the average percent reduction in delivery tours by using the CDPP relative to each benchmark for Cook, Adams, and Cumberland counties. Including the search time for parking reduces the completion time of delivery  tours in all counties. The CDPP reduces the completion time up to 53\% on average relative to the no-parking-time benchmark. The CDPP also outperforms industry practice and models in the literature. Using the Modified TSP to model real-life practice, the CDPP reduces the completion time up to 11\% on average. The Relaxed M-S benchmark generalizes the current models in the literature. Figure \ref{BenchmarkComparison_BaseCase} shows that the Relaxed M-S benchmark with $\alpha = 0.6$ performs similarly to the no-parking-time benchmark. Increasing $\alpha$ to 0.8, the CDPP reduces the completion time up to 48\% on average. In Cumberland County, the CDPP realizes greater reductions at higher levels of $\alpha$. Insight \ref{Parking Matters} summarizes this result.

\begin{insight}\label{Parking Matters}
Parking matters in last-mile delivery optimization. The CDPP outperforms industry practice and models in the literature highlighting the value of determining the order of service and including the search time for parking in optimal routing decisions.
\end{insight}

\begin{figure}[h]
  \centering
  \includegraphics[scale = 0.3]{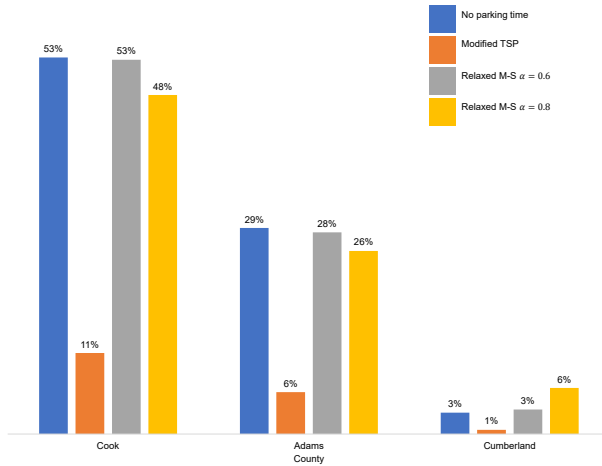}
  \caption{Average percent reduction in completion time of delivery tours by using CDPP relative to the no-parking-time benchmark, Modified TSP, Relaxed M-S with $\alpha = 0.6$, and Relaxed M-S with $\alpha = 0.8$ in the base case. }
\label{BenchmarkComparison_BaseCase}
\end{figure}

The impact of parking differs across customer geographies. The CDPP provides the greatest savings in Cook County, an urban environment with the highest customer density and search time for parking. Including the search time for parking in the CDPP reduces the completion time of these delivery tours by an average of 53\% relative to when no parking time is considered. Further, the delivery person saves an average of 11\% and 48\% in delivery time relative to the Modified TSP and Relaxed M-S with $\alpha = 0.8$, respectively. In rural areas where the customer density and  search time for parking is generally low, including the search time for parking reduces the completion time of the delivery tour for the CDPP on average 3\% in Cumberland County relative to when no parking time is considered. The CDPP provides a similar completion time in the delivery tour as the Modified TSP. Therefore, using the TSP solution may be sufficient in making routing decisions for rural environments. Increasing the value of $\alpha$ in Equation \eqref{MS_generalized_obj} makes walking advantageous in the solution to the Relaxed M-S benchmark. However, in rural environments, when the distance between customers is greater, a solution that incentivizes walking adversely affects the realized completion time of the delivery tour. Therefore, we see a higher level of savings for the CDPP when $\alpha = 0.8$ than $\alpha = 0.6$. We further explore this result later in this section. Insight \ref{BetterUrban} summarizes these results. 

\begin{insight} \label{BetterUrban}
Including the search time for parking in routing optimization for last-mile delivery provides the greatest advantage in urban environments where parking is a challenge. In rural areas, the TSP may be sufficient in making routing decisions. 
\end{insight}

Now, we explore how the structure of an optimal solution to the CDPP differs from the solutions of the benchmark problems. Figure \ref{RoutingStructures} shows the solutions to the (a) CDPP, (b) no-parking-time benchmark, and (c) Modified TSP for a portion of an instance in Cook County. We focus on a portion of the solution to highlight the local differences in the CDPP and benchmark solutions. The solid black lines indicate the driving path of the vehicle. The dotted lines indicate the walking paths of the delivery person. Each color represents a different service set. The flag icons indicate the parking spots. The numerical label at the customer location indicates the order service within the tour. For example, a numerical label of $i$ indicates that this customer is $i$-th on the tour. Figure \ref{RoutingStructures}a shows an optimal solution for the CDPP in this instance of Cook County. The delivery person serves three customer sets from a single parking spot. Figure \ref{RoutingStructures}b shows the solution when no parking time is considered (i.e. $p=0$). For this instance, the solution to the no-parking-time benchmark is 
a TSP solution, parking at every customer, and driving between customer locations. However, the  search time for parking in Cook County is 9 minutes. If the delivery person follows the solution in Figure \ref{RoutingStructures}b, the delivery person spends about an hour looking for parking to serve these eight customers. Figure \ref{RoutingStructures}a shows that even though walking is slower than driving, it is more advantageous to park once and serve all customers on foot. Insight \ref{CompareIgnoreParkingTime} summarizes this observation. 

\begin{insight} \label{CompareIgnoreParkingTime}
When parking time is ignored, routing decisions that focus on the fastest way to service all customers may result in unnecessary time spent searching for parking. 
\end{insight}

\begin{figure}[h]
\centering
\begin{subfigure}{0.49\textwidth}
  \centering
  \includegraphics[scale =0.3] {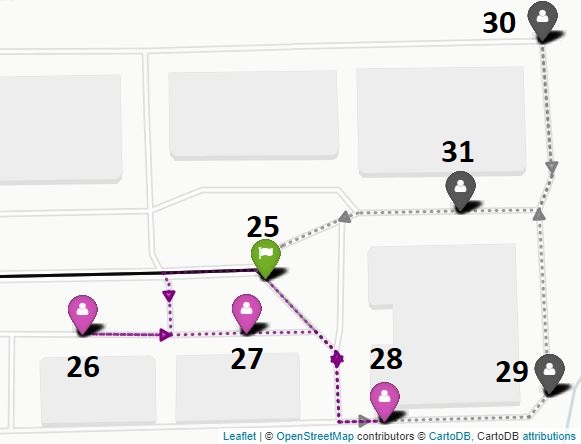}
  \caption{CDPP}
\end{subfigure}

\begin{subfigure}{.49\textwidth}
  \centering
  \includegraphics[scale =0.3] {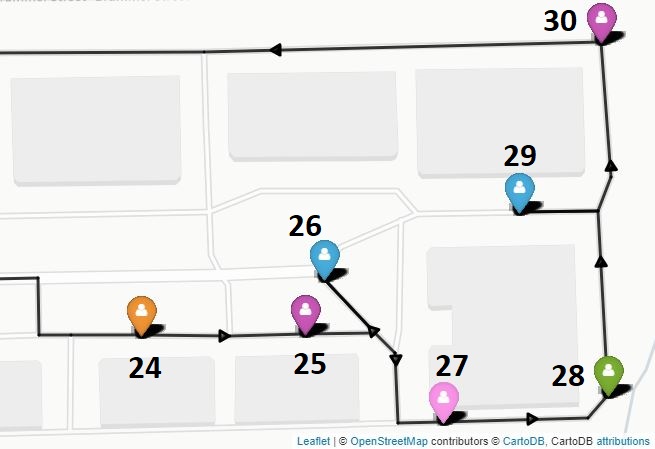}
  \caption{No Parking Time}
\end{subfigure}
\begin{subfigure}{.49\textwidth}
\centering
  \includegraphics[scale = 0.3]{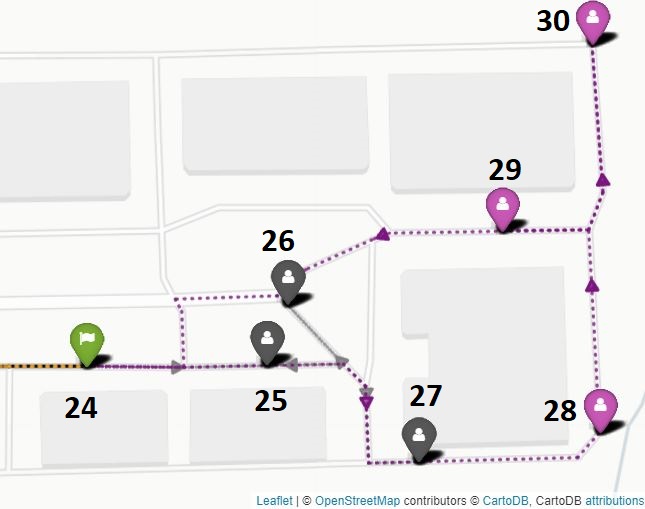}
  \caption{Modified TSP}
\end{subfigure}
\caption{\boldmath Solutions to a portion of an instance of Cook County for the (a) CDPP, (b) no-parking-time benchmark, and (c) Transformed TSP benchmark in the base case.}
\label{RoutingStructures}
\end{figure}

Recall that the Modified TSP uses a solution to the TSP to fix the order of customer service. Figure \ref{RoutingStructures}c shows the solution of the Modified TSP benchmark. The Modified TSP optimizes the trade-offs between walking, driving, and searching for parking. Similar to the CDPP solution in Figure \ref{RoutingStructures}a, Figure \ref{RoutingStructures}c shows that the delivery person parks once to avoid high search times for parking and walks to service customers on foot. However, fixing the order of service to the TSP results in additional walking time for the delivery person relative to the CDPP in Figure \ref{RoutingStructures}a. The CDPP takes into account the search time for parking when determining where to park as well as the walking path for the delivery person reducing the completion time of the delivery tour relative to the Modified TSP. Insight \ref{TransformedTSPComparison} summarizes this observation.

\begin{insight} \label{TransformedTSPComparison}
The CDPP outperforms the Modified TSP by optimizing the service order which leads to better trade-offs between driving, walking, and searching for parking.
\end{insight}

Finally, we analyze the solution structure of the Relaxed M-S benchmarks to understand the impact of using the search time for parking in the objective function of the CDPP. Figure \ref{MS_CDPP_Breakdown} shows the average time (in minutes) spent in the Relaxed M-S and CDPP delivery tours searching for parking, driving, walking, and loading packages for Cook and Cumberland counties. When $\alpha = 0.6$, the solution structure in both counties relies on the delivery person driving with limited to no walking. A solution that focuses on driving forces the delivery person to search for many parking locations. In Cook County, Figure \ref{MS_CDPP_Breakdown}a shows the delivery person spends on average 77\% of the delivery tour searching for parking. Increasing $\alpha$ to 0.8 makes walking advantageous and reduces the completion time of the delivery tour by 11\% on average. However, searching for parking remains a significant portion of the delivery tour. Including the search time for parking in the objective function further reduces the completion time of the delivery tour by 48\% on average. In Cook County, where the customer density is high, the delivery person controls the total search time for parking by parking in fewer locations and walking to service customers. When customers are further apart, like Cumberland County, Figure \ref{MS_CDPP_Breakdown}b shows that increasing $\alpha$ to 0.8 increases the completion time of the delivery tour 3\% on average. Making walking advantageous in rural environments forces the delivery person to walk between customers that are further apart. The similarity in the solution structures between the Relaxed M-S with $\alpha = 0.6$ and the CDPP support the conclusions of Insight \ref{BetterUrban} that a solution that relies on driving is sufficient in rural environments. Insight \ref{MSComparison} summarizes these results. 

\begin{insight} \label{MSComparison}
Including the search time for parking in the objective for last-mile delivery routing is critical to achieve optimal trade-offs between driving, walking, and searching for parking.
\end{insight}

\begin{figure}[h]
\centering
\begin{subfigure}{.5\textwidth}
  \centering
  \includegraphics[scale =0.25] {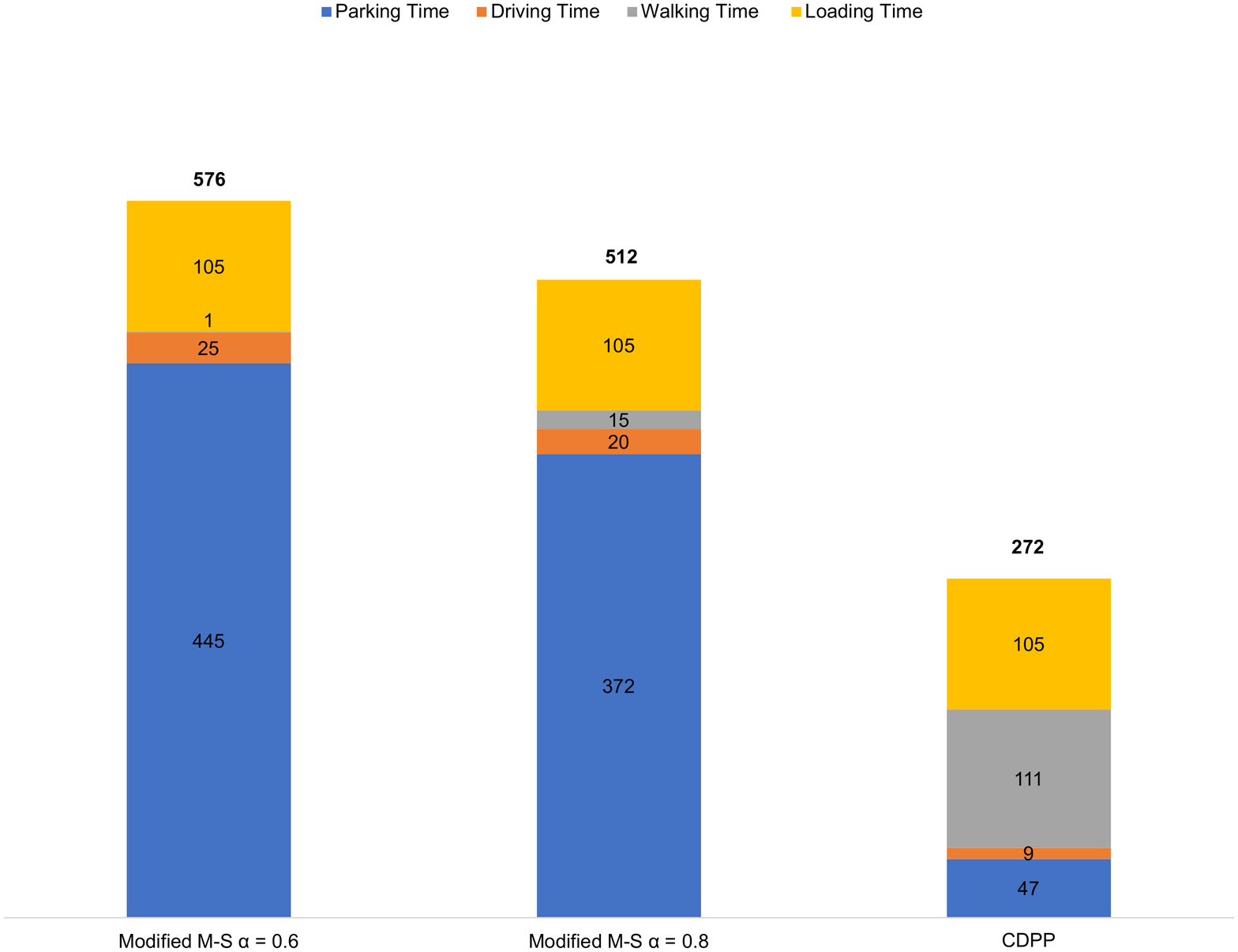}
  \caption{Cook County}
\end{subfigure}%
\begin{subfigure}{.5\textwidth}
  \centering
  \includegraphics[scale = 0.25]{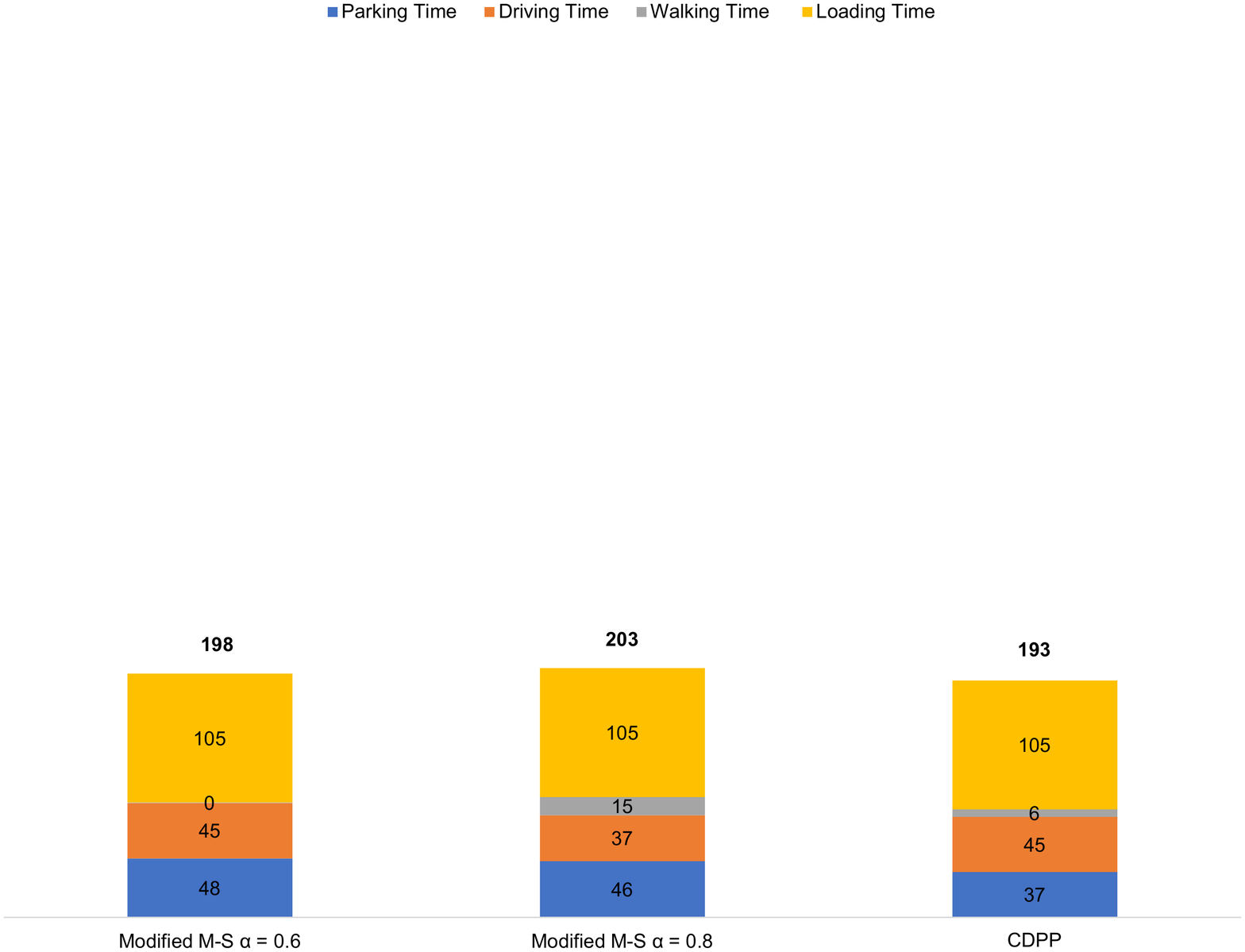}
  \caption{Cumberland County}
\end{subfigure}
\caption{Average time (minutes) spent in the Modified M-S and CDPP delivery tours searching for parking, driving, walking, and loading packages for (a) Cook County and (b) Cumberland County in the base case.}
\label{MS_CDPP_Breakdown}
\end{figure}

\subsection{Impact of search time for parking} \label{ParkingTimeP}

In this section, we focus on how the search time for parking $p$ changes the structure to the CDPP solution. Figure \ref{Cook_q3_Breakdown} shows the average time (in minutes) spent in optimal CDPP delivery tours searching for parking, driving, walking, and loading packages for Cook County in the base case with various parking times. In urban settings, the search time for parking is expected to be high. We test $p=0, 3, 6,$ and $9$ minutes. Figure \ref{Cook_q3_Breakdown} shows that the total time spent searching for parking remains relatively stable --- between 40-47 minutes on average --- when $p>0$, indicating that the delivery person parks fewer times as the search time for parking increases. Relatedly, we observe the walking time significantly increases (57 to 111 minutes on average) as the delivery person must walk further to service more customers. Insight \ref{ParkingTimeInsight} summarizes this observation.

\begin{insight} 
\label{ParkingTimeInsight}
In urban areas, as search time to find parking increases, the total time spent looking for parking remains stable. At higher search times for parking, the delivery person parks fewer times at the expense of significantly increasing walking time.
\end{insight}

\begin{figure}[h]
  \centering
  \includegraphics[scale = 0.3]{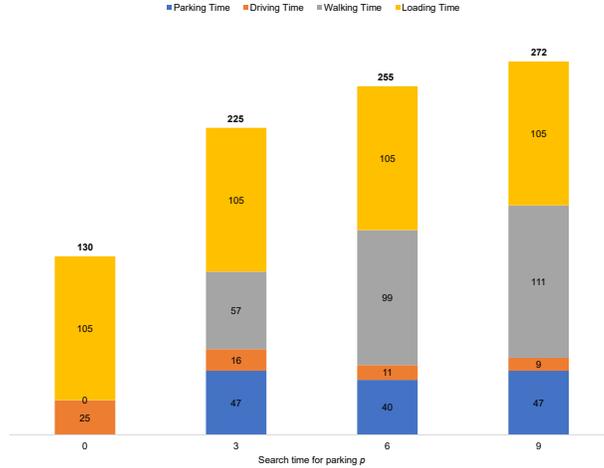}
  \caption{Average time (minutes) spent in optimal CDPP delivery tours searching for parking, driving, walking, and loading packages for Cook County in the base case with various parking times. }
\label{Cook_q3_Breakdown}
\end{figure}

Outside of urban environments, we observe that differences in customer geography reduce the impact of the search time for parking on the structure of the CDPP solution. Figure \ref{Adams_Cumberland_q3_Breakdown} shows the average time (in minutes) spent in optimal CDPP delivery tours searching for parking, driving, walking, and loading packages for Adams and Cumberland counties. Recall that Adams and Cumberland counties are classified as suburban and rural areas, respectively. Therefore, we expect customers to be further apart in these counties than in an urban area, like Cook County. When $p$ increases from 0 to  3 minutes, the completion time of the delivery tour increases by approximately the same amount on average in Cook and Adams counties (95 minutes in Cook County and 98 minutes in Adams County). However, the impact on the structure of the solution differs. When $p$ increases from 0 to 3 minutes in Cook County, Figure \ref{Cook_q3_Breakdown} shows the average search time for parking increases 47 minutes accounting for 49\% of the increase in the completion time of the delivery tour. The remaining 51\% of the increase reflects changes in the solution structure, i.e. less driving time and more walking time for the delivery person. In Adams County, when $p$ increases from 0 to 3 minutes, Figure \ref{Adams_Cumberland_q3_Breakdown}a shows that the average search time for parking increases 65 minutes. In this case, the total search time for parking accounts for 66\% of the increase in the completion time of the delivery tour. Similar to Cook County, when $p$ increases from 0 to 3 minutes, Figure \ref{Adams_Cumberland_q3_Breakdown}a shows driving time decreases and walking time increases, but these changes in the solution structure only account for 33\% of the increase in the completion time of the delivery tour. Therefore, the increase in the search time for parking has less impact on the solution structure in Adams County than Cook County. The differences in how the solution changes between Cook and Adams counties reflect that the CDPP solution cannot trade off more walking at a higher search time for parking in a county, such as Adams, where the customers are relatively further apart. Insight \ref{CapacityGeography} summarizes these results.

\begin{insight} \label{CapacityGeography}
Differences in customer geography influence the significance of including the search time for parking in routing decisions. Increasing the search time for parking outside of urban environments has less impact on the solution structure than in urban environments. 
\end{insight}

\begin{figure}[h]
\centering
\begin{subfigure}{.5\textwidth}
  \centering
  \includegraphics[scale =0.25] {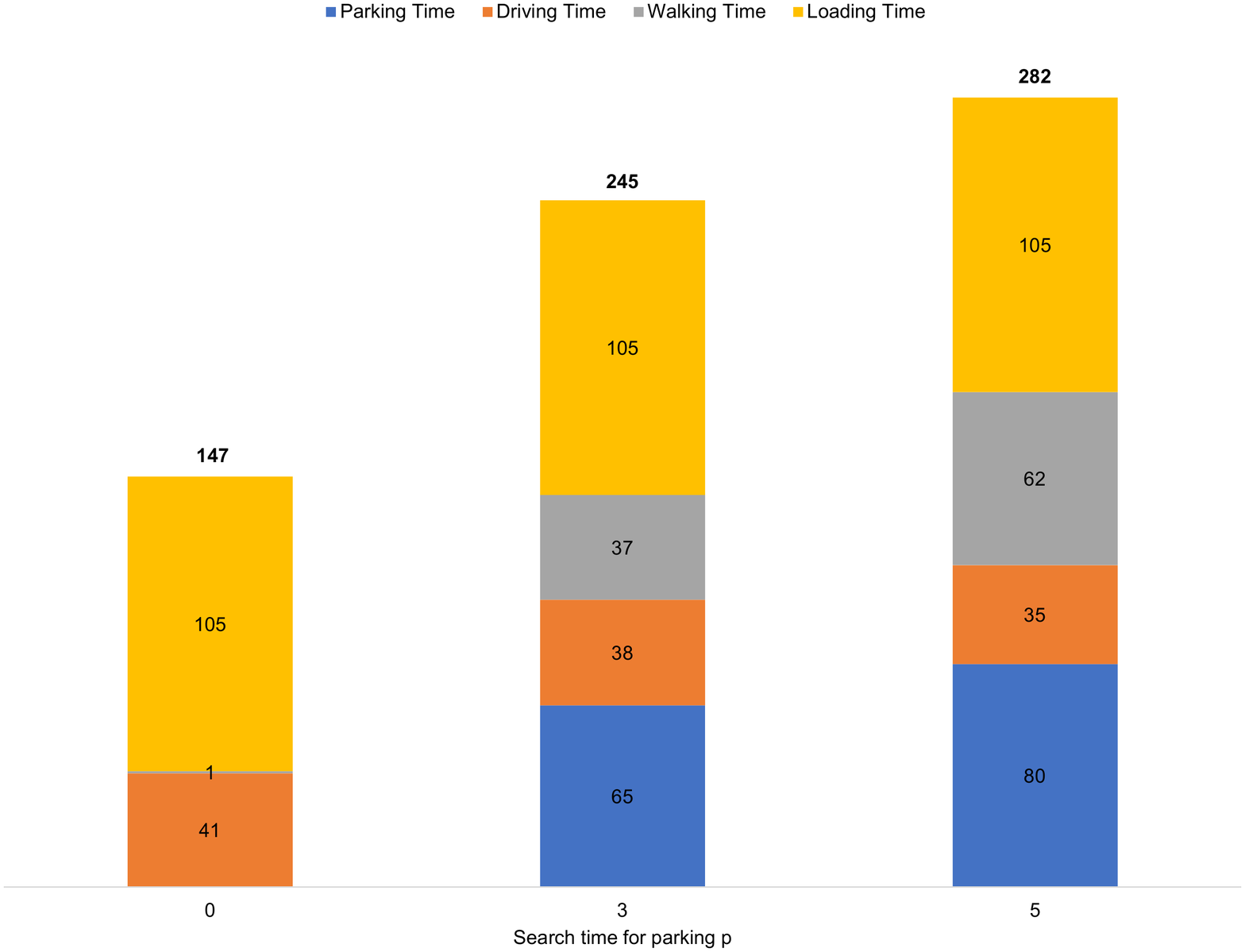}
  \caption{Adams County}
\end{subfigure}%
\begin{subfigure}{.5\textwidth}
  \centering
  \includegraphics[scale = 0.25]{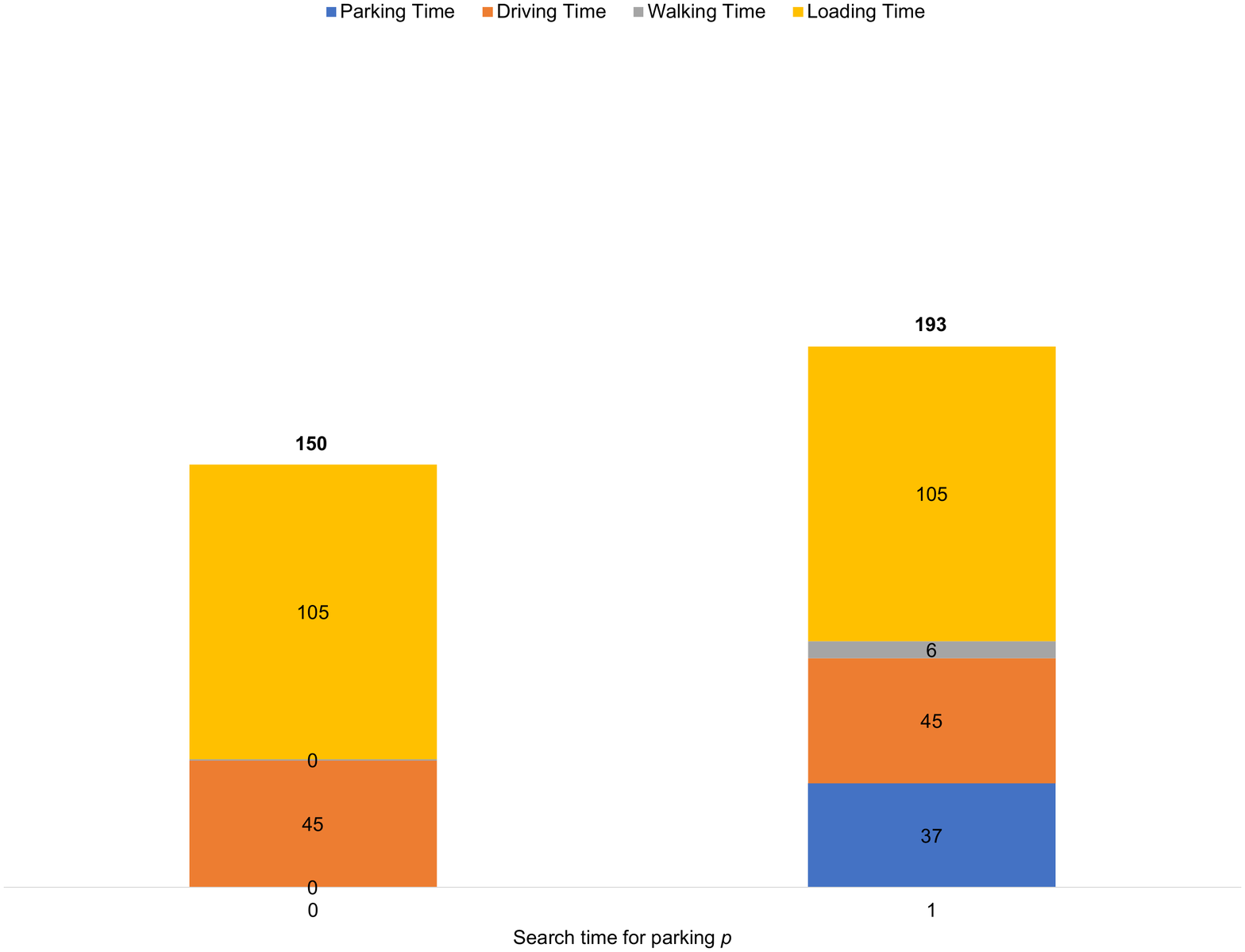}
  \caption{Cumberland County}
\end{subfigure}
\caption{Average time (minutes) spent in optimal CDPP delivery tours searching for parking, driving, walking, and loading packages for (a) Adams County and (b) Cumberland County in the base case with various parking times.}
\label{Adams_Cumberland_q3_Breakdown}
\end{figure}

In rural environments, like Cumberland County, we test $p=0$ and 1 minute as we expect search time for parking to be low. Since customers are likely further apart than urban and suburban environments, we expect the delivery person to spend more time driving between customers than walking. Figure \ref{Adams_Cumberland_q3_Breakdown}b shows that there exists customer locations close enough together such that it is advantageous to walk to them even when $p=1$ minute. However, the driving time remains the same on average indicating that the increase in the search time for parking does not have a significant effect on routing decisions. These observations support the conclusion of Insight \ref{BetterUrban} that the TSP driving to all customers may be sufficient in making routing decisions for rural environments.

\subsection{Impact of the Capacity of Delivery Person} \label{CapacityQ}

In this section, we discuss the impact of the capacity of the delivery person $q$ on reducing the completion time of the delivery tour in different customer geographies. Figure \ref{Capacity_Comparison} shows the average objective value of the CDPP for Cook, Adams, and Cumberland counties in the base case with varying capacities. In urban areas, like Cook County, increasing the capacity from $q=1$ to 4 packages reduces the completion time of the delivery tour by 20\% on average. We observe a smaller impact in Adams County with an average reduction of 6\% when increasing from $q=1$ to 4 packages. In rural areas, like Cumberland County, we observe little difference between the solution for $q=1$ and $4$ packages. This observation supports Insight \ref{BetterUrban} showing that serving customers individually may be sufficient for routing decisions in rural areas. For all counties, we observe marginally decreasing reductions in the delivery tour when increasing capacity. Insight \ref{CapacityInsight} summarizes these results. 

\begin{insight} \label{CapacityInsight}
Increasing the capacity of the delivery person is more advantageous outside of rural areas. In all customer geographies, marginal reductions in the completion time of the delivery tour decrease at higher capacities.
\end{insight}

\begin{figure}[h]
  \centering
  \includegraphics[scale = 0.3]{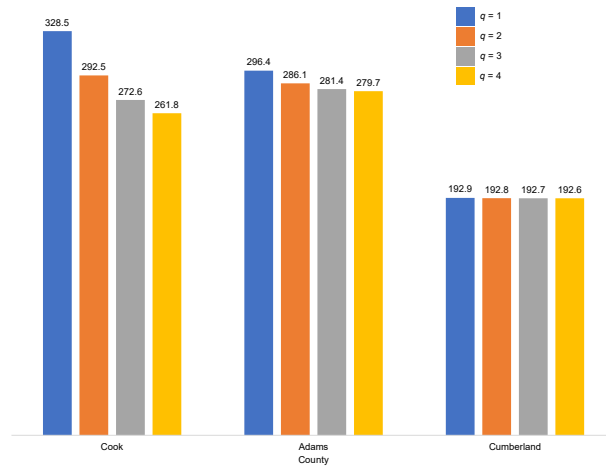}
  \caption{Average optimal value of the CDPP for Cook, Adams, and Cumberland Counties in the base case with varying capacities.}
\label{Capacity_Comparison}
\end{figure}

\section{Heuristic Solution to CDPP} \label{Heuristic}

As we have shown in Section \ref{Computational Performance}, our model improvements allow us to efficiently solve problems with $n=50$ customers and $q=3$ packages. Most problems in the related literature focus on at most $n=50$ customers and support the use of $q=3$ packages as a base case. If we want to solve instances with more customers or larger capacities for the delivery person, we face computational limitations due to the growth in the size of the model. For example, when $n=50$ customers, the average runtime for $q=4$ packages is 5.9, 4.4, and 3.1 hours for Cook, Adams, and Cumberland counties, respectively. Increasing to $q=5$ packages, the number of $y_{ij}$ variables increases to 118,496,750 making the problem intractably large. Increasing from $n=50$ to $100$ also significantly increases the runtime. For example, when $n=100$ customers, the average runtime for $q=2$ packages is 0.4, 3.1, and 8.5 hours for Cook, Adams, and Cumberland counties, respectively. For $n=100$, with the base case of $q=3$ packages, the problem again becomes computationally infeasible.

Section \ref{Experimental Results} shows that by including the search time for parking in routing optimization for last-mile delivery, the solution to the CDPP may signifcantly reduce the completion time of the delivery tour. To realize these benefits for larger instances, we provide a heuristic solution to the CDPP that finds high quality solutions quickly. The proposed two-echelon location-routing heuristic decomposes the decisions of the CDPP into two echelons. Section \ref{FirstEchelon} describes the first echelon, defined as the parking assignment and routing problem (PA-R), where the customers are assigned to parking locations and the route of the vehicle between these parking locations is defined. For each parking spot, Section \ref{SecondEchelon} describes the second echelon, defined as the service set assignment problem (SSA), which determines how to optimally partition the customers in service sets given the parking spots, therefore, defining the walking paths of the delivery person. Then, Section \ref{Heuristic_Results} discusses the quality of the heuristic solutions.

\subsection{Parking Assignment and Routing Problem} \label{FirstEchelon}

The PA-R determines where to park the vehicle, the assignment of customers to parking locations, and the route of the vehicle between these parking locations. We further decompose the PA-R into two IPs. 

First, we determine where to park the vehicle and assign customers to parking locations. Let $\hat{p}_{i} = 1$ if the delivery person parks at customer $i$ for $i \in C$, $0$ otherwise. Let $\hat{a}_{ik} = 1$ if customer $k$ is assigned to parking spot $i$ for $i,k \in C$, $0$ otherwise. All service times used in the IPs of this section are defined in Section \ref{ServiceTimes}. The cost of opening a parking spot is the search time for parking $p$. Therefore, the PA-R captures parking at fewer locations when the search time for parking is high. Solve the following MIP and denote an optimal solution $Z$. 

\begin{align} \label{Heuristic_obj_simple}
\min & \sum_{i \in C} p \cdot \hat{p}_{i} + \sum_{i \in C} \sum_{k \in C} W(i,k) \cdot \hat{a}_{ik} \\ \label{HeuristicAssignAll_simple}
\textrm{s.t.} &  \sum_{k \in C} \hat{a}_{ki} = 1 && \forall i \in C \\ \label{HeuristicVisit_simple}
& \hat{a}_{ik} \leq \hat{p}_{i} && \forall i,k \in C \\ \label{Heuristic_open_simple}
& \hat{p}_i \leq \sum_{k \in C} \hat{a}_{ik} && \forall i \in C \\
\label{Heuristic_p_simple}
& \hat{p}_{i} \in \{0,1\} && \forall i \in C \\ \label{Heuristic_a_simple} 
& \hat{a}_{ik} \in \{0,1\} && \forall i,k \in C
\end{align}

\noindent The objective function in Equation \eqref{Heuristic_obj_simple} minimizes a linear combination of the search time for parking and the assignment of walking (without return walks) to the customers from parking locations. Constraints \eqref{HeuristicAssignAll_simple} require that each customer is assigned to a parking spot. Given that a customer is assigned to a parking spot, Constraints \eqref{HeuristicVisit_simple} require the vehicle to be parked at that parking spot. If no customers are assigned to a parking spot, then Constraints \eqref{Heuristic_open_simple} ensure that the parking spot is not opened. Finally, Constraints \eqref{Heuristic_p_simple} and \eqref{Heuristic_a_simple} give the binary constraints on variables $\hat{p}_{i}$ and $\hat{a}_{ik}$, respectively.

Now, let $P = \{i\in C| \hat{p}_{i} = 1 \textrm{ in } Z \}$ be the parking spots in solution $Z$. We find the TSP solution of $P \cup 0$ with respect to the driving times, $D(i,k)$, using the standard TSP IP formulation with single commodity subtour elimination constraints \citep{Subtour}. This solution provides the route of the vehicle between the parking locations. 

\subsection{Service Set Assignment Problem} \label{SecondEchelon}

For each parking spot, the SSA partitions the customers into customer service sets defining the walking paths of the delivery person. For each $i \in P$, define $K_i = \{k \in C|\hat{a}_{ik} = 1 \textrm{ in } Z\}$ to be the customers assigned to parking spot $i$ in solution $Z$. For each $i \in P$, define $S_i \subset S$ to be the potential sets to be served at $i$, i.e. $S_i = \{\sigma_j \in S| \sigma_j \subseteq K_i\}$. For each $k \in K_i$, let $\hat{J}_k = \{\sigma_j \in J_k | \sigma_j \in S_i\}$ be the service sets in $S_{i}$ that include customer $k$. Let $\hat{y}_j =1$ if $\sigma_j$ is serviced for $\sigma_j \in S_i$. Solve the following MIP for each parking spot $i\in P$ to determine the service sets. 

\begin{align} \label{ChooseSets_obj}
\min & \sum_{\sigma_j \in S_i} w_{ij}\hat{y}_j \\ \label{ChooseSets_VisitAll}
\textrm{s.t.} & \sum_{j \in \hat{J}_k} \hat{y}_j =1 && \forall k \in K_i \\ \label{ChooseSets_y}
& \hat{y}_j \in \{0,1\} && \forall j \in S_i
\end{align}

\noindent The objective function in Equation \eqref{ChooseSets_obj} minimizes the walking time for the delivery person servicing customers $K_i$ while parked at customer $i$. Constraints \eqref{ChooseSets_VisitAll} require each customer in $K_i$ to be in a service set. Constraints \eqref{ChooseSets_y} give the binary constraints on the variables $\hat{y}_j$.

\subsection{Quality of Heuristic Solutions} \label{Heuristic_Results}

In this section, we evaluate the quality of the two-echelon location-routing heuristic. For this analysis, we consider the base case (location-dependent parking times and $f=2.1$ minutes) for $n=50$ and $100$ customers with various capacities of the delivery person $q$. For $n=50$, averages are taken across ten instances. For $n=100$, averages are taken across five instances. Appendix \ref{Heuristic_details} provides detailed results on the objective value of the heuristic solution.

First, we discuss the quality of the solution with respect to the optimal value of the CDPP. Table \ref{Heuristic_OptimalityGap} provides the average optimality gap between the heuristic and optimal solutions for the CDPP across counties and capacities $q$. An asterisk (*) by the county name indicates that the CDPP MIP of one instance in that county was solved to an optimality gap of 1.2\%. When $n=100$ and $q\geq 3$, we face computational limitations when solving the CDPP and,  therefore, the optimality gap in this case cannot be evaluated. Otherwise,  the value of the heuristic solution is on average within 5.5\% of the CDPP optimal value. For each county, the heuristic performs best for $q=2$ packages and worst for $q=1$ package. 

\begin{table}[h]
\renewcommand{\arraystretch}{0.75}
\centering
\begin{tabular}{lrrrr}
\toprule
& \multicolumn{4}{c}{\boldmath $q$} \\
\cmidrule(lr){2-5}
\textbf{County} & \multicolumn{1}{c}{\textbf{1}} & \multicolumn{1}{c}{\textbf{2}} & \multicolumn{1}{c}{\textbf{3}} & \multicolumn{1}{c}{\textbf{4}}\\
\midrule 
 Cook & &&&\\
 \quad $n=50$ & 5.5\% & 1.3\% & 3.4\% & 5.5\%\\
 \quad $n=100$ & 4.9\% & 0.5\% & - & - \\
 \hline
 Adams & &&&\\
 \quad $n=50$& 4.0\% & 2.2\% & 3.1\% & 3.6\%\\
\quad $n=100$ & 4.9\% & 2.6\% & - & - \\ 
\hline  
 Cumberland &&&&\\
 \quad $n=50$ & 1.6\% & 1.5\% & 1.5\% & 1.5\% \\
\quad $n=100$* & 1.2\% & 1.1\% & - & - \\
\bottomrule
\end{tabular}
\caption{\boldmath Average optimality gap (\%) between the two-echelon location-routing heuristic and optimal solution for the CDPP across various capacities of the delivery person $q$ in the base case.} \label{Heuristic_OptimalityGap}
\end{table}

Next, we discuss the runtime for the two-echelon location-routing heuristic. On average, the PA-R takes at most 0.6 minutes and the SSA takes at most 0.3 minutes. In total, the average runtime of the two-echelon location-routing heuristic is at most 0.7 minutes. Thus, the two-echelon location-routing heuristic finds high quality solution quickly.

Section \ref{Experimental Results} shows that including the search time for parking in routing optimization by using the CDPP for last-mile delivery is advantageous in reducing the completion time of the delivery tour. We show that the two-echelon location-routing heuristic outperforms the benchmark problems in Section \ref{Benchmarks} for most cases, highlighting that this heuristic provides an improvement to industry practice and models in the literature at low computational time. Table \ref{Benchmarks_OptimalityGap} shows the average optimality gap between the benchmark models in Section \ref{Benchmarks} and the optimal CDPP value for $n=50$ across various capacities of the delivery person $q$ in the base case. A single asterisk $(^*)$ indicates that 9 out of 10 instances for the benchmark model being considered solved to optimality. A double asterisk $(^{**})$ indicates that 8 out of 10 instances for the benchmark model being considered solved to optimality. We also include the optimality gap for the two-echelon location routing heuristic. For each county and capacity, the model that achieved the lowest optimality gap is bolded. The two-echelon location-routing heuristic significantly outperforms all benchmarks in Cook and Adams counties. For Cumberland County, the two-echelon location-routing heuristic solution outperforms all benchmarks aside from the Modified TSP supporting the conclusion that driving the TSP and parking at every customer may be a sufficient solution in rural environments. This analysis supports the use of the two-echelon location-routing heuristic to realize the advantages of including the search time for parking in routing last-mile delivery.

\begin{table}[h]
\renewcommand{\arraystretch}{0.75}
\centering
\begin{tabular}{lrrr@{}lr@{}l}
\toprule
& \multicolumn{6}{c}{\boldmath \textbf{$q$}} \\
\cmidrule(lr){2-7}
\textbf{County} & \multicolumn{1}{c}{\textbf{1}} & \multicolumn{1}{c}{\textbf{2}} &\multicolumn{2}{c}{\textbf{3}} & \multicolumn{2}{c}{\textbf{4}}\\
\midrule 
 Cook & &&&&&\\
 \quad No parking time & 77.0\% & 98.8\% & 113.3\% && 122.1\% &\\
 \quad Modified TSP & 13.9\% &  13.7\% & 12.9\% & & 12.8\%& \\
 \quad Relaxed M-S $\alpha = 0.6$ & 75.6\% & 97.8\% & 111.7\% && 121.1\% &\\
  \quad Relaxed M-S $\alpha = 0.8$ & 65.6\% & 79.9\% & 93.0\%& & 100.0\% &$^*$\\
  \quad Two-Echelon Location-Routing Heuristic & \textbf{5.5\%} & \textbf{1.3\%} &\textbf{ 3.4\%} & & \textbf{5.5\%} &\\
\hline
 Adams & &&&&&\\
 \quad No parking time & 34.8\% & 39.9\% & 42.4\% && 43.4\% &\\
 \quad Modified TSP & 7.4\% & 7.0\% & 6.4\%&& 5.7\%&\\
  \quad Relaxed M-S $\alpha = 0.6$ & 33.6\% & 38.0\% & 41.0\% & & 40.8\% & \\
 \quad Relaxed M-S $\alpha = 0.8$ & 31.0\% & 33.7\% & 36.3\% &$^*$ & 33.7\% & $^{**}$\\
  \quad Two-Echelon Location-Routing Heuristic & \textbf{4.0\%} & \textbf{2.2\%} & \textbf{3.1\%} & & \textbf{3.6\%} &\\
 \hline
 Cumberland &&&&\\
 \quad No parking time & 4.1\% & 4.2\% & 4.2\% && 4.2\% &\\
  \quad Modified TSP & \textbf{0.9\%} & \textbf{0.7\%} & \textbf{0.6\%}&& \textbf{0.6\%} &\\
  \quad Relaxed M-S $\alpha = 0.6$ & 3.2\% & 3.5\% & 3.1\% && 3.2\% &\\
 \quad Relaxed M-S $\alpha = 0.8$ & 6.1\% & 7.0\% & 6.9\% &$^*$ & 8.0\%& $^{**}$ \\
 \quad Two-Echelon Location-Routing Heuristic & 1.6\% & 1.5\% & 1.5\% && 1.5\% &\\
\bottomrule
\end{tabular}
\caption{Average optimality gap (\%) between the benchmark models in Section \ref{Benchmarks} and the optimal CDPP value for $n=50$ across various capacities of the delivery person $q$ in the base case.} \label{Benchmarks_OptimalityGap}
\end{table}

 \section{Conclusions and Future Work} \label{Conclusions_FutureWork}

Yes, parking matters. Including the search time for parking in the objective for routing optimization for last-mile delivery may significantly reduce the completion time of the delivery tour. Using the CDPP for routing decisions provides the greatest advantage in urban environments where parking is a challenge and often time consuming. When parking is a challenge, the CDPP solution recommends that the delivery person park in fewer locations and serve multiple customer service sets from the same parking spot. This decision balances the trade-offs of walking to service customers and driving to find a new parking location. However, in rural environments where parking is more readily available, a solution that parks at all customers serving each customer individually may be sufficient. 

These insights are reflected in our analytical results that determine when the search time for parking becomes large enough that the TSP solution parking at every customer location is not an optimal solution to the CDPP. These results show that in urban environments, where customers are close together, small search times for parking impact the structure of the optimal solution in the CDPP. A TSP solution that parks at every customer best approximates optimal solutions to the CDPP in rural environments where customers are further apart.

To solve reasonably-sized instances of the CDPP, this paper introduces several valid inequalities and a variable reduction technique that improves computational performance of the CDPP. However, further work needs to be done to control the growth in the number of variables in the model, particularly the service variables $y_{ij}$ which grows both in the number of customers as well as the number of potential service sets. A column generation approach may improve computational performance. For instances where the model becomes intractably large, we propose a two-echelon location-routing heuristic and show that this heuristic finds high quality solutions quickly. These heuristic solutions outperform other traditional last-mile delivery models providing an immediate improvement to industry practice and models in the literature. 

This work provides immediate ways to improve routing a single vehicle for last-mile delivery. Future work includes considering the impact of parking under additional delivery constraints, such as customer time windows and additional service times, as well as the consideration of a fleet of vehicles. In the experimental design, we assume that the parking locations are restricted to customer locations. However, it may be the case that the vehicle cannot park at every customer and/or there exists parking locations outside of customer locations that should be considered. Future work includes the analysis of further restricting or increasing parking locations in the CDPP.

We model last-mile delivery in a deterministic framework to build insights on the impact that the search time for parking has on routing decisions. However, we know parking is stochastic and can vary by location and time of day. In Seattle, \cite{Cruise_for_parking} conclude that the cruise time for parking decreases when more curb-space is allocated to commercial loading zones and paid parking and increases when more curb-space is allocated to bus zones. A survey of 16 drivers in New York City finds that the search time for parking ranges from 3 minutes in Brooklyn to 60 minutes in Midtown East \citep{Survey_Drivers}. Future work includes the analysis of spatial changes in the search time for parking. In addition, future work includes generalizing the CDPP to consider temporal changes in search time for parking as well as the impact of parking in a stochastic framework. 

An understanding of the availability of parking locations and its impact on delivery practices also benefits urban planning efforts. Currently, many urban areas are experimenting with loading zone reservations and pricing schemes for delivery vehicles \citep{DC_curbs, BostonParking}. For example, a pilot program in Aspen, Colorado, allows commercial drivers to reserve and pay for the use of ``Smart Zones" with a mobile app \citep{AspenParking}. Insights into how parking impacts routing in last-mile delivery may allow for better placement of loading zones and better pricing schemes to incentivize drivers and yield better curbside management.

\bibliographystyle{plainnat}
\bibliography{References}

%
%
%

\newpage 
 \begin{appendices}
 
 \section{Proofs} \label{Proofs}
 
 \subsection{Proof of Claim \ref{Grid_q2}}
 \begin{proof}
To conclude part (a), we begin by providing a lower bound on a solution to the CDPP where the delivery person parks $k$ times for $k < n$. Then, we show this bound is bounded below by the objective value for the TSP solution where the delivery person parks at every customer, given in Equation \eqref{TSP_obj_value_eq}. Thus, the TSP solution parking at every customer location is an optimal solution to the CDPP. 

The CDPP objective value in Equation \eqref{CDPPobj} decomposes into the time the delivery person spends searching for parking, driving, walking, and loading packages. Assume the delivery person parks $k$ times. Then, the search time for parking is $kp$. \cite{AutonomousGrid} show that the closest customer to the depot is unique and $D(0,c_2) = MinDistance + 1$ for each customer $c_2$ in the set of second closest customers to the depot. When $k>1$, driving time to and from the depot is minimized if the delivery person enters the grid at the closest customer and exits the grid at a second closest customer to the depot. Thus, driving time to and from the grid is bounded below by $(2\cdot MinDistance +1)\hat{d}\hat{l}$ minutes. The case where $k=1$ is addressed at the end of this discussion. Further, the driving time between the $k$ parking spots is bounded below by $(k-1)\hat{d}\hat{l}$ minutes as the parking locations are at least one unit apart. In total, the delivery person drives at least $(2\cdot MinDistance + k)\hat{d}\hat{l}$ minutes. Since the delivery person parks $k$ times, $n-k$ customers are serviced from a parking spot that is not the customer location. When $q\leq 2$, the delivery person walks at least $2\hat{w}\hat{l}$ minutes per customer as the customer is at least one block away from the parking location and must also return back to the parking location. In total, the delivery person walks at least $2\hat{w}\hat{l}(n-k)$ minutes. In summary, Equation \eqref{LB_q12} gives a lower bound on a solution where the delivery person parks $k$ times,
\begin{align} \label{LB_q12}
kp + (2\cdot MinDistance + k)\hat{d}\hat{l} + 2\hat{w}\hat{l}(n-k) + nf.
\end{align}
We show that the lower bound in Equation \eqref{LB_q12} holds for $k=1$. Note that when $k=1$, driving time to and from the single parking spot is bounded below by $(2\cdot MinDistance)\hat{d}\hat{l}$ minutes. However, for at least one customer, the delivery person must walk at least $4w$ units. Therefore, the lower bound when $k=1$ is given in the following equation: 
\begin{align} \nonumber 
& p + (2\cdot MinDistance)\hat{d}\hat{l} + 2\hat{w}\hat{l}(n-2) + 4\hat{w}\hat{l} + nf \\ \label{k_1_a}
& \hspace{10em}\geq p + (2\cdot MinDistance)\hat{d}\hat{l} + 2\hat{w}\hat{l}(n-1) + 2\hat{w}\hat{l} + nf \\ \label{k_1_b}
& \hspace{10em}\geq p + (2\cdot MinDistance + 1)\hat{d}\hat{l} + 2\hat{w}\hat{l}(n-1)+ nf. 
\end{align}
Equation \eqref{k_1_a} redistributes the $4\hat{w}\hat{l}$ term. Equation \eqref{k_1_b} follows from the assumption $\hat{w}>\hat{d}$, or equivalently $2\hat{w}\hat{l}>2\hat{d}\hat{l}>\hat{d}\hat{l}$. Thus, the lower bound in Equation \eqref{LB_q12} holds for $k=1$, as well. 

Now, we show the lower bound in Equation \eqref{LB_q12} is bounded below by the objective value in Equation \eqref{TSP_obj_value_eq}. By adding $0 = (n-k)\hat{d}\hat{l} - (n-k)\hat{d}\hat{l}$ to Equation \eqref{LB_q12}, it follows
\begin{align} \nonumber
 &kp + (2\cdot MinDistance + k + (n-k) - (n-k))\hat{d}\hat{l} + 2\hat{w}\hat{l}(n-k) + nf \\ \label{LB_q12_a}
 & \hspace{10em} \leq kp + (n-k)p + (2\cdot MinDistance + n)\hat{d}\hat{l} + nf \\ \label{LB_q12_b} 
&\hspace{10em} = np+ (2\cdot MinDistance + n)\hat{d}\hat{l} + nf.
\end{align}
Equation \eqref{LB_q12_a} uses the assumption that $p \leq \hat{l}(2\hat{w}-\hat{d})$, or equivalently $p(n-k) \leq 2\hat{w}\hat{l}(n-k)-\hat{d}\hat{l}(n-k)$. Then, Equation \eqref{LB_q12_b} simplifies Equation \eqref{LB_q12_a}. Lemma \ref{TSP_obj_value} concludes that Equation \eqref{LB_q12_b} is the objective value to the TSP solution where the delivery person parks at every customer location. Thus, the lower bound to a solution where the delivery person parks $k$ times is bounded below by this TSP solution and we conclude an optimal solution to the CDPP is the TSP solution where the delivery person parks at every customer location.

To conclude part (b), we construct a solution that has a lower objective value than the TSP solution where the delivery person parks at every customer location. Thus, this TSP solution is not an optimal solution to the CDPP. 

First, we construct a solution to the CDPP. Figure \ref{6_6_Grid_ConstructedSols}(a) shows an example of this constructed solution on a $6 \times 6$ grid of customers. Dashed green lines represent walking paths for the delivery person. Consider the complete grid of customers $\sqrt{n} \times \sqrt{n}-1$ with the bottom left corner of the grid at $(1,1)$, bottom right corner at $(\sqrt{n},1)$, top left corner at $(1, \sqrt{n}-1)$, and top right corner at $(\sqrt{n}, \sqrt{n}-1)$. Since $\sqrt{n}(\sqrt{n}-1)$ is even, \cite{AutonomousGrid} show there exists a Hamiltonian path from $(1,1)$ to $(1, 2)$ through the solid rectangular grid. In Figure \ref{6_6_Grid_ConstructedSols}(a), this path is shown with the solid, blue driving route. The delivery person parks at each customer traversed on this path. The remaining $\sqrt{n}$ customers are located on the line from $(1, \sqrt{n})$ to $(\sqrt{n}, \sqrt{n})$. For each $a \in \{1,2,...,\sqrt{n}\}$, the customer located at $(a, \sqrt{n})$ is serviced from the parking spot $(a,\sqrt{n}-1)$. Equation \eqref{ObjValue_ConstructedSolution_q2} gives the objective value for this constructed solution
\begin{align} \nonumber 
&(2\cdot MinDistance + 1 +\sqrt{n}(\sqrt{n}-1)-1) \hat{d}\hat{l} + \sqrt{n}(\sqrt{n}-1)p+ 2\hat{w}\hat{l}\sqrt{n} +nf\\ \label{ObjValue_ConstructedSolution_q2}
&\hspace{10em} =(2\cdot MinDistance + n - \sqrt{n}) \hat{d}\hat{l} + (n-\sqrt{n})p+ 2\hat{w}\hat{l}\sqrt{n} + nf.
\end{align}
Then, it follows 
\begin{align} \nonumber
&(2\cdot MinDistance + n - \sqrt{n}) \hat{d}\hat{l} + (n-\sqrt{n})p+ 2\hat{w}\hat{l}\sqrt{n} + nf \\ \label{Const_Sol_q2_1}
&\hspace{10em} < (2\cdot MinDistance + n) \hat{d}\hat{l} + (n-\sqrt{n})p + nf + p\sqrt{n} \\ \label{Const_Sol_q2_2}
&\hspace{10em} =(2\cdot MinDistance + n) \hat{d}\hat{l} + np + nf.
\end{align}
Equation \eqref{Const_Sol_q2_1} uses the assumption $p>\hat{l}(2\hat{w}-\hat{d})$, or equivalently $p\sqrt{n}>2\hat{w}\hat{l}\sqrt{n}-\hat{d}\hat{l}\sqrt{n}$. Then, Equation \eqref{Const_Sol_q2_2} simplifies Equation \eqref{Const_Sol_q2_1}. Lemma \ref{TSP_obj_value} concludes that Equation \eqref{Const_Sol_q2_2} is the objective value to the TSP solution where the delivery person parks at every customer location. Thus, this constructed solution has a lower objective value than the TSP solution where the delivery person parks at every customer location concluding that this TSP solution is not an optimal solution to the CDPP. 
\end{proof}

\begin{figure}[h]
  \centering
\begin{subfigure}{.5\textwidth}
\centering
  \includegraphics[scale = 0.4]{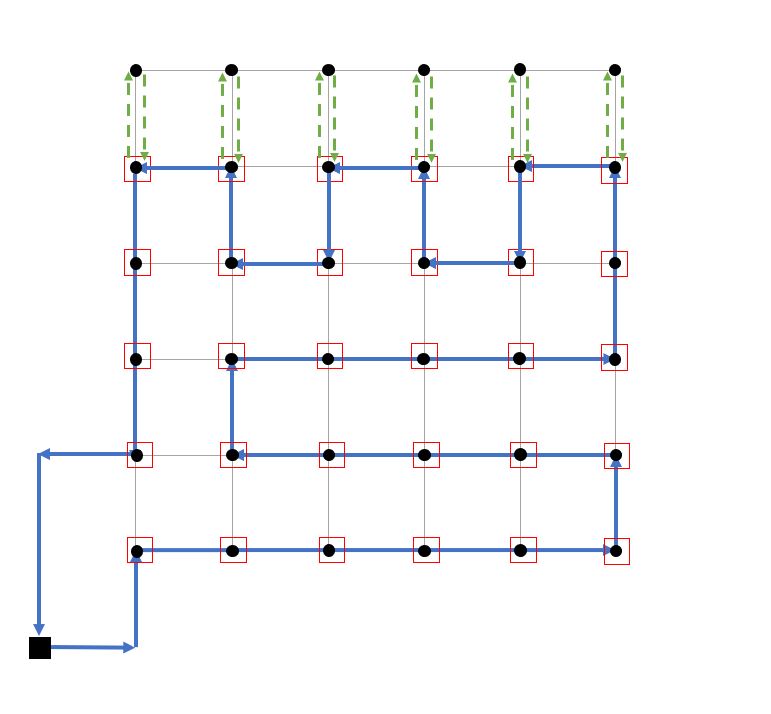}
  \caption{$q\leq2$ and $p>\hat{l}(2\hat{w}-\hat{d})$}
\end{subfigure}%
\begin{subfigure}{.5\textwidth}
  \centering
  \includegraphics[scale=0.4]{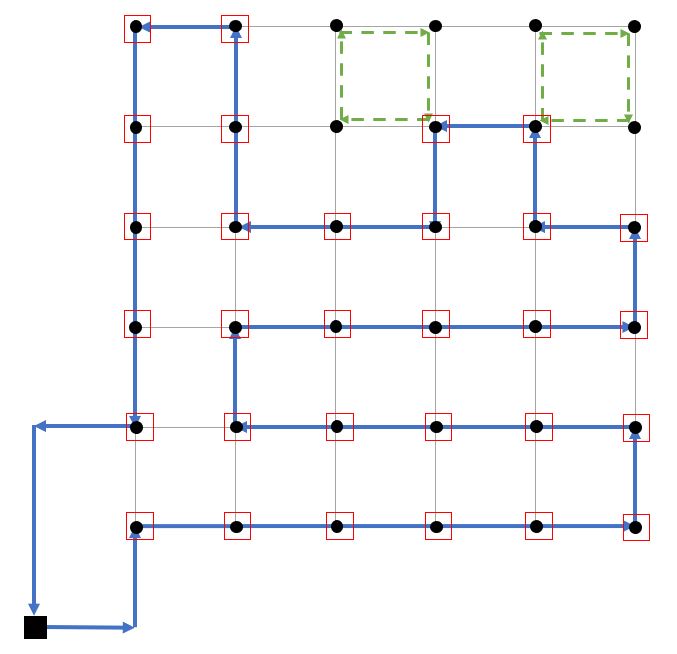}
\caption{$q=3$ and $p>\hat{l}(\frac{4}{3}\hat{w}-\hat{d})$} 
 \end{subfigure}
  \caption{Constructed solutions that achieve a lower objective value than the TSP solution parking at every customer when (a) $q\leq2$ and $p>\hat{l}(2\hat{w}-\hat{d})$  or (b) $q=3$ and $p>\hat{l}(\frac{4}{3}\hat{w}-\hat{d})$ on a $6 \times 6$ grid of customers. }
\label{6_6_Grid_ConstructedSols}
\end{figure}

\subsection{Proof of Claim \ref{Grid_q3}}

\begin{proof}
This proof takes the same approach outlined in the proof of Claim \ref{Grid_q2}.

To conclude part(a), we begin by providing a lower bound on a solution to the CDPP where the delivery person parks $k$ times for $k<n$. The search time for parking is $kp$. The arguments for the lower bound on the driving time presented in the proof of Claim \ref{Grid_q2}(a) hold here. The delivery person drives at least $(2\cdot MinDistance + k)\hat{d}\hat{l}$ minutes. Since the delivery person parks $k$ times, $n-k$ customers are serviced from a parking spot that is not the customer location. If the delivery person serves one or two customers in a service set, then the delivery person walks at least $2\hat{w}\hat{l}$ minutes per customer as the customer is at least one block away from the parking location and must also return back to the parking location. If the delivery person serves three customers in a service set, the delivery person walks at least $4\hat{w}\hat{l}$ minutes to walk to each customer and return back to the parking location. To find a lower bound on the walking time, we maximize the consolidation efforts to minimize the number of return walks to parking locations. Find $g,h \in \mathbb{Z}$ where $h<3$ and $n-k = 3g + h$. Then, the delivery person walks at least $4g + 2h$ units. In summary, Equation \eqref{LB_q3} gives a lower bound on a solution where the delivery person parks $k$ times,
\begin{align} \label{LB_q3}
kp + (2\cdot MinDistance + k)\hat{d}\hat{l} + (4g+2h)\hat{w}\hat{l}+ nf.
\end{align}
We show that the lower bound in Equation \eqref{LB_q3} holds for $k=1$. Note that when $k=1$, driving time to and from the single parking spot is bounded below by $(2\cdot MinDistance)\hat{d}\hat{l}$ minutes. However, for at least one set of three customers, the parking location is at least one unit away so the delivery person must walk at least $6w$ units. Therefore, the lower bound when $k=1$ is given in the following equation: 
\begin{align} \label{k_1_a_q3} 
p + (2\cdot MinDistance)\hat{d}\hat{l} + 6w\hat{w}\hat{l} + 4(g-1)\hat{w}\hat{l} + 2h\hat{w}\hat{l} + nf 
\end{align}
Equation \eqref{k_1_a_q3} is still bounded below by Equation \eqref{LB_q3}. Thus, the lower bound in Equation \eqref{LB_q3} holds for $k=1$, as well. 

Now, we show the lower bound in Equation \eqref{LB_q3} is bounded below by the objective value in Equation \eqref{TSP_obj_value_eq}. By adding $0 = (n-k)\hat{d}\hat{l} - (n-k)\hat{d}\hat{l}$ to Equation \eqref{LB_q3}, it follows
\begin{align} \nonumber
&kp + (2\cdot MinDistance + k + (n-k) - (n-k))\hat{d}\hat{l} + (4g+2h)\hat{w}\hat{l}+ nf \\ \label{lb_TSP_1}
&\hspace{4em} = kp + (2\cdot MinDistance + k + (n-k) - (3g+h))\hat{d}\hat{l} + (4g+2h)\hat{w}\hat{l}+ nf \\\label{lb_TSP_2}
& \hspace{4em} \geq kp + (2\cdot MinDistance + k + (n-k) - (3g+h))\hat{d}\hat{l} + (4g+\frac{4}{3}h)\hat{w}\hat{l}+ nf \\ \label{lb_TSP_3}
& \hspace{4em} \geq kp + (2\cdot MinDistance + k + (n-k) - (3g+h))\hat{d}\hat{l} + nf + (3g+h)p + (3g + h)\hat{d}\hat{l} \\ \label{lb_TSP_4}
& \hspace{4em} = kp + (2\cdot MinDistance + n)\hat{d}\hat{l} + nf + (3g+h)p \\ \label{lb_TSP_5}
& \hspace{4em} = np + (2\cdot MinDistance + n)\hat{d}\hat{l} + nf.
\end{align}
Equation \eqref{lb_TSP_1} uses that $n-k = 3g+h$. Then, Inequality \eqref{lb_TSP_2} uses that $2h\hat{w}\hat{l}\geq \frac{4}{3}h\hat{w}\hat{l}$. Inequality \eqref{lb_TSP_3} uses the assumption that $p \leq \hat{l}(\frac{4}{3}\hat{w}-\hat{d})$, or equivalently $3gp + 3g\hat{d}\hat{l} \leq 4g\hat{w}\hat{l}$ and $hp + h\hat{d}\hat{l} \leq \frac{4}{3}h\hat{w}\hat{l}$. Then, Equation \eqref{lb_TSP_4} simplifies Equation \eqref{lb_TSP_3}. Finally, Equation \eqref{lb_TSP_5} simplifies Equation \eqref{lb_TSP_4} using $n-k = 3g+h$. Lemma \ref{TSP_obj_value} concludes that Equation \eqref{lb_TSP_5} is the objective value to the TSP solution where the delivery person parks at every customer location. Thus, the lower bound to a solution where the delivery person parks $k$ times is bounded below by this TSP solution and we conclude an optimal solution is the TSP solution where the delivery person parks at every customer location.

To conclude part (b), we construct a solution that has a lower objective value than the TSP solution where the delivery person parks at every customer location. Algorithm \ref{q=3 Algorithm} constructs a particular solution when $q=3$ packages and Figure \ref{6_6_Grid_ConstructedSols}(b) shows an example of this constructed solution on a $6 \times 6$ grid of customers. This solution parks at $n-6$ customers for a total search time for parking of $p(n-6)$. As discussed in Steps \ref{DriveOver1Walk} and \ref{DriveOver2Walk} in Algorithm \ref{q=3 Algorithm}, the delivery person serves two sets of three customers for a total walking time of $8\hat{l}\hat{w}$. In total, the delivery person drives $2\cdot MinDistance + n - 6$ units for a total driving time of $(2\cdot MinDistance + n - 6)\hat{l}\hat{d}$. Equation \eqref{Constructed_Sol_obj_q3} gives the objective value for this constructed solution
\begin{align} \label{Constructed_Sol_obj_q3}
(2\cdot MinDistance + n - 6)\hat{l}\hat{d} + (n-6)p + 8\hat{l}\hat{w} + nf.
\end{align}
Then, it follows 
\begin{align} \nonumber
&(2\cdot MinDistance + n - 6)\hat{l}\hat{d} + (n-6)p + 8\hat{l}\hat{w} + nf\\ \label{Const_Sol_q3_1}
&\hspace{10em}<(2\cdot MinDistance + n - 6)\hat{l}\hat{d} + (n-6)p + nf + 6p + d\hat{l}\hat{d} \\ \label{Const_Sol_q3_2}
&\hspace{10em}= (2\cdot MinDistance + n)\hat{l}\hat{d} + np + nf.
\end{align}
Equation \eqref{Const_Sol_q3_1} uses the assumption $p>\hat{l}(\frac{4}{3}\hat{w}-\hat{d})$, or equivalently $6p + 6d >8\hat{w}\hat{l}$. Then, Equation \eqref{Const_Sol_q3_2} simplifies Equation \eqref{Const_Sol_q3_1}. Lemma \ref{TSP_obj_value} concludes that Equation \eqref{Const_Sol_q3_2} is the objective value to the TSP solution where the delivery person parks at every customer location. Thus, this constructed solution has a lower objective value than the TSP solution where the delivery person parks at every customer location concluding that this TSP solution is not an optimal solution to the CDPP. 
\end{proof}

\begin{algorithm}[h]
\caption{Constructed Solution on $\sqrt{n} \times \sqrt{n}$ complete grid when $q=3$} \label{q=3 Algorithm}
\algsetup{indent = 2em}
\begin{algorithmic} [1]
\STATE \textbf{Input:} Number of customers $n = \sqrt{n}\times \sqrt{n}$ where $\sqrt{n}$ is even
\STATE \textbf{Output:} Feasible Solution for CDPP when $q=3$

\STATE \label{DriveToGrid} Drive from depot at $(0,0)$ and park at customer location $(1,1)$.
\STATE \label{DrivePark} Drive and park at each customer location from $(1,1)$ to $(\sqrt{n},1)$.
\STATE \label{ZigZag}
for $i \in \{1,..., \frac{\sqrt{n}-4}{2}\}$: \\
\hspace{0.5cm} Drive and park at customer location $(\sqrt{n}, 2i)$. \\
\hspace{0.5cm} Drive and park at each customer location from $(\sqrt{n},2i)$ to $(2,2i)$. \\
\hspace{0.5cm} Drive and park at customer location $(2, 2i+1)$. \\
\hspace{0.5cm} Drive and park at each customer location from $(2,2i+1)$ to $(\sqrt{n},2i+1)$.
\STATE \label{DriveUp}
Drive and park from customer location $(\sqrt{n},n-3)$ to customer location $(\sqrt{n},n-2)$.
\STATE \label{DriveOver1}
Drive and park at customer location $(\sqrt{n}-1,n-2)$.
\STATE \label{DriveOver1Walk}
Drive and park at customer location $(\sqrt{n}-1,\sqrt{n}-1)$. Serve this customer individually. Then, from this parking location, walk to service the customer service set $\{(\sqrt{n}-1,\sqrt{n}), (\sqrt{n},\sqrt{n}), (\sqrt{n},\sqrt{n}-1)\}$.
\STATE \label{DriveOver2Walk}
Drive and park at customer location $(\sqrt{n}-2,\sqrt{n}-1)$. Serve this customer individually. Then, from this parking location, walk to service the customer service set $\{(\sqrt{n}-2,\sqrt{n}), (\sqrt{n}-3,\sqrt{n}), (\sqrt{n}-3,\sqrt{n}-1)\}$.
\STATE \label{DriveDown}
Drive and park at customer location $(\sqrt{n}-2, \sqrt{n}-2)$.
\STATE \label{DriveLeft}
Drive and park at customer location $(\sqrt{n}-3, \sqrt{n}-2)$.
\STATE \label{ZigZag2}
for $i \in \{1,..., \frac{\sqrt{n}-4}{2}\}$: \\
\hspace{0.5cm} Drive and park at customer location $(\sqrt{n}-4+2i, \sqrt{n}-2)$. \\
\hspace{0.5cm} Drive and park at each customer location from $(\sqrt{n}-4+2i, \sqrt{n}-2)$ to $(\sqrt{n}-4+2i,\sqrt{n})$. \\
\hspace{0.5cm} Drive and park at customer location $(\sqrt{n}-5+2i, \sqrt{n})$. \\
\hspace{0.5cm} Drive and park at each customer location from $(\sqrt{n}-5+2i, \sqrt{n})$ to $(\sqrt{n}-5+2i, \sqrt{n}-2)$. \\
\STATE \label{LastDown} Drive and park at each customer location from $(1, \sqrt{n}-2)$ to $(1,1)$. \\
\STATE \label{DriveToDepot} Drive from customer location $(1,1)$ to depot at $(0,0)$. 
\end{algorithmic}
\end{algorithm}

\subsection{Proof of Claim \ref{park_set}}

\begin{proof}
Assume the delivery person parks at $i \in \Pi \cap C$ but serves $i$ while parked at a different parking location $l \in \Pi \setminus\{i\}$. Therefore, there exists $\sigma_j \in J_i$ such that $y_{lj} = 1$. Take $v = |\sigma_j|$ and denote ($c_1,...,c_v$) to be an optimal order to service set $\sigma_j$ when parked at $l$. Note that $c_I = i$ for some $I \in \{1,...,v\}$.  After serving customer $c_v$, the delivery person returns to parking spot $l$ making a cycle. Now, we consider two cases: $l \in \sigma_j$ and $l \notin \sigma_j$. If $l \in \sigma_j$, an optimal order to serve set $\sigma_j$ when parked at $i$ is ($c_I, c_{I+1},...,c_v,l,c_1,..., c_{I-1}$). In this case, $w_{ij} = w_{lj}$ giving an equivalent solution. If $l \not \in \sigma_j$, it follows 
\begin{align} \label{park_set_1}
w_{lj} &= W(l, c_1) + W(c_1,c_2)+\cdots +W(c_{I-1}, c_I)+W(c_I,c_{I+1})+\cdots + W(c_{v-1}, c_v) + W(c_v, l) \\\label{park_set_2}
&\geq W(c_1,c_2)+\cdots +W(c_{I-1}, c_I)+W(c_I,c_{I+1})+\cdots + W(c_{v-1}, c_v) + W(c_v, c_1) \\ \label{park_set_3}
& = W(c_I,c_{I+1})+\cdots + W(c_{v-1}, c_v) + W(c_v, c_1) +W(c_1,c_2)+\cdots +W(c_{I-1}, c_I) \\ \label{park_set_4}
&\geq w_{ij}.
\end{align}
Equation \eqref{park_set_1} follows by the definition of the service times in Section \eqref{ServiceTimes}. Inequality \ref{park_set_2} holds by the triangle inequality on the walking time. Equation \eqref{park_set_3} rearranges the terms and Inequality \eqref{park_set_4} follows by the definition of service times in Section \ref{ServiceTimes}. Thus, serving set $\sigma_j$ from parking spot $l$ is an upper bound on the solution that serves $\sigma_j$ from parking spot $i$. In conclusion, if the delivery person parks at $i \in \Pi \cap C$, then there exists $\sigma_j \in J_i$ such that $y_{ij} = 1$. 
\end{proof}

\subsection{Proof of Claim \ref{f0_Single}}

\begin{proof}
Assume $x_{ki}=1$ for some $k \in \Pi \setminus \{i\}$ but $y_{ij_i} = 0$ in an optimal solution. We will show there exists an equivalent solution to the CDPP such that $y_{ij_i} =1$. Since $x_{ki} = 1$, Claim \ref{park_set} concludes there exists $\sigma_j \in J_i$ such that $y_{ij} = 1$. By assumption, $\sigma_j \neq \sigma_{j_i}$. Define $\sigma_l = \sigma_j \setminus \{i\}$. Observe that $\sigma_j = \sigma_{j_i} \cup \sigma_l$. By the definition of $w_{ij}$ in Section \ref{ServiceTimes}, it follows $w_{ij} = w_{ij_i} + w_{il}$. Since $f_j$ is linearly dependent on the number of packages, it follows $f_j = f_{j_i} + f_l$. In addition, $\sigma_{j_i}$ and $\sigma_{j_l}$ satisfy any capacity constraints as subsets of $\sigma_j$. Therefore, the objective value of the solution that services $\sigma_j$ when parked at $i$ is equivalent to the solution that services $\sigma_{j_i}$ and $\sigma_l$ when parked at $i$. 
\end{proof}

\subsection{Proof of Corollary \ref{f0_y0}}

\begin{proof}
Let $\bar{J} = \{\sigma_j \in J_i\vert \textrm{ } |\sigma_j| \geq 2\}$. We consider two cases: the vehicle parks at $i$ and the vehicle does not park at $i$. If the vehicle parks at customer $i$, Claim \ref{f0_Single} concludes $y_{ij_i} = 1$. By Constraints \eqref{CDPPMember}, it follows $y_{kj} = 0$ for all $(k,j) \in \Pi \times J_i$ except $(i,j_i)$. Since $j_i \notin \bar{J}$, $y_{ij} = 0$ for all $\sigma_j \in \bar{J}$. If the vehicle does not park at $i$, then $\sum_{k \in \Pi} x_{ki} = 0$. By Constraints \eqref{CDPPVisit}, $y_{ij} = 0$ for all $\sigma_j \in S$. Since $\bar{J} \subset S$, $y_{ij} = 0$ for all $\sigma_j \in \bar{J}$. 
\end{proof}

\section{General Model Improvements} \label{AdditionalModelImprovements}

In this appendix, we present model improvements that apply to the CDPP with any loading time function $f_j$. The results of Section \ref{ModelImprovements} are restricted to when the loading time $f_j$ is linearly dependent on the number of packages in the service set. First, we note that Claim \ref{park_set} in Section \ref{ModelImprovements} holds for any loading time function $f_j$. Here, we provide additional general results  and point out that these results are strengthened in Section \ref{ModelImprovements} when considering a specific loading time function (i.e. $f_j = f \cdot |\sigma_j|$). 

We strengthen the relationship between the $x_{ik}$ variables and $y_{ij}$ variables by identifying the purpose of the parking location. The delivery person parks the vehicle to service customers on foot. If the delivery person parks at parking location $i \in \Pi \setminus \{0\}$, Claim \ref{ParkServe} shows a set of customers is serviced while parked at customer $i$ in an optimal solution.  

\begin{claim} \label{ParkServe}
If the delivery person parks at customer $i \in \Pi \setminus \{0\}$ (i.e. $x_{ki}=1$ for some $k \in \Pi \setminus \{i\}$), then there exists $\sigma_j \in S$ such that $y_{ij}=1$ in an optimal solution, i.e. 
\begin{align} \label{ParkServeVI}
x_{ki} \leq \sum_{\sigma_j\in S} y_{ij} && \forall i \in \Pi \setminus \{0\}, k \in \Pi \setminus \{i\}.
\end{align}
\end{claim}

\begin{proof}
Assume $x_{ki}=1$ but $y_{ij}=0$ for all $\sigma_j \in S$. By Constraints \eqref{CDPPComeLeave}, there exists $l \in \Pi \setminus \{k\}$ such that $x_{il}=1$. Therefore, the following contributes to the objective function in Equation \eqref{CDPPobj}:
\begin{align} \label{ParkServe1}
d_{ki}+d_{il} &= D(k,i) + p + D(i,l) + p \\ \label{ParkServe2}
&\geq D(k,i)+D(i,l) + p \\ \label{ParkServe3}
&\geq D(k,l) + p \\ \label{ParkServe4}
&= d_{kl}.
\end{align}
Equation \eqref{ParkServe1} follows from the definition of the service times in Section \ref{ServiceTimes}. Since $p\geq 0$, Inequality \eqref{ParkServe2} holds. Since the driving time satisfies the triangle inequality, we arrive at Inequality \eqref{ParkServe3} which by definition of the service times in Section \ref{ServiceTimes} is Equation \eqref{ParkServe4}. We conclude that an optimal solution is bounded below by a solution that drives directly from $k$ to $l$ without parking at $i$. Since $y_{ij}=0$ for all $\sigma_j \in S$, the change in the solution structure maintains the feasibility of the solution. Thus, Equations \eqref{ParkServeVI} hold for an optimal solution. 
\end{proof}

Claim \ref{ParkServe} concludes the delivery person services at least one customer set from each parking spot. A key feature of the CDPP is that the delivery person can serve multiple customer sets from the same parking spot. Therefore, the number of parking spots is less than or equal to the number of services sets in an optimal solution. Corollary \ref{parkingspots} formalizes this observation.

\begin{corollary} \label{parkingspots}
The number of parking spots is less than or equal to the number of service sets in an optimal solution, i.e.
\begin{align} \label{ParkLessThanServe}
\sum_{i \in \Pi} \sum_{k \in \Pi\setminus \{0,i\}} x_{ik} \leq \sum_{i \in \Pi\setminus \{0\}} \sum_{\sigma_j \in S} y_{ij}.
\end{align}
\end{corollary}

\section{Detailed Experimental Results} \label{Heuristic_details}

Table \ref{Heuristic_Objective Value} shows the average value of the two-echelon location-routing heuristic  across various capacities of the delivery person $q$ in the base case (location-dependent parking times and $f=2.8$ minutes). For $n=50$, averages are taken across ten instances. For comparison purposes, Figure \ref{Capacity_Comparison} shows the average optimal CDPP value for $n=50$ and $q=1$ to 4 packages. For $n=100$, averages are taken across five instances. For comparison purposes, Table \ref{CDPP_n100_Objective Value} shows the  average optimal CDPP value for $n=100$ and $q=1$ to 2 packages. 

\begin{table}[h]
\renewcommand{\arraystretch}{0.75}
\centering
\begin{tabular}{lcccccc}
\toprule
& \multicolumn{6}{c}{\boldmath \textbf{$q$}} \\
\cmidrule(lr){2-7}
\textbf{County} & \textbf{\boldmath $1$} & \textbf{\boldmath $2$} & \textbf{\boldmath $3$} & \textbf{\boldmath $4$} & \textbf{\boldmath $5$} & \textbf{\boldmath $6$}\\
\midrule 
 Cook & &&&&&\\
 \quad $n=50$ & 346.4 & 296.2 & 281.8 & 276.2 & 273.9 & 272.4\\
 \quad $n=100$ & 651.4 & 550.6 & 526.1 & 513.3 & 507.9 & 505.2 \\
\hline
 Adams & &&&&&\\
 \quad $n=50$& 308.8 & 293.0 & 290.5 & 290.3 & 290.2 & 290.2\\
\quad $n=100$ & 564.4 & 528.6 & 523.2 & 521.7 & 521.7 & 521.3 \\ 
 \hline
 Cumberland &&&&&&\\
 \quad $n=50$ & 196.0 & 195.5 & 195.5 & 195.5 & 195.5 & 195.5 \\
\quad $n=100$ & 364.9 & 364.1 & 364.0 & 364.0 & 364.0 & 364.0 \\
\bottomrule
\end{tabular}
\caption{Average value of the two-echelon location-routing heuristic  across various capacities of the delivery person $q$ in the base case.} \label{Heuristic_Objective Value}
\end{table}

\begin{table}[h]
\renewcommand{\arraystretch}{0.75}
\centering
\begin{tabular}{lcc}
\toprule
& \multicolumn{2}{c}{\boldmath \textbf{q}} \\
\cmidrule(lr){2-3}
\textbf{County} & \textbf{\boldmath $1$} & \textbf{\boldmath $2$} \\
\midrule 
 Cook & 620.8 & 547.6 \\
Adams & 537.4 & 514.2  \\
Cumberland* & 360.5 & 359.9  \\
\bottomrule 
\end{tabular}
\caption{Average optimal value of the CDPP for $n=100$ across various capacities of the delivery person $q$ in the base case.} \label{CDPP_n100_Objective Value}
\end{table}

 \end{appendices}

\end{document}